%% file: main.tex
\documentclass{article} 
\usepackage{times}

\input{math_commands.tex}

\usepackage{amsmath} 
\usepackage{amsthm} 
\usepackage{enumitem}
\usepackage{bbm}

\usepackage{todonotes}

\usepackage[utf8]{inputenc} 
\usepackage[T1]{fontenc}    
\usepackage{float}
\usepackage{hyperref}       

\usepackage{url}            
\usepackage{booktabs}       
\usepackage{cellspace}
\usepackage{amsfonts}       
\usepackage{nicefrac}       
\usepackage{microtype}      
\usepackage{xcolor}         
\usepackage{breqn}

\usepackage{comment}

\usepackage{natbib}
\usepackage{amssymb}
\usepackage{algorithm}
\usepackage{algpseudocode}
\usepackage{graphicx}
\usepackage{xfrac}
\usepackage{multirow}
\usepackage{pifont}

\usepackage[capitalize]{cleveref} 
\usepackage{aliascnt}

\newtheorem{theorem}{Theorem}[section]
\crefname{theorem}{theorem}{theorems}
\Crefname{theorem}{Theorem}{Theorems}

\newaliascnt{lemma}{theorem}
\newtheorem{lemma}[lemma]{Lemma}
\aliascntresetthe{lemma}
\crefname{lemma}{lemma}{lemmas}
\Crefname{lemma}{Lemma}{Lemmas}

\newaliascnt{corollary}{theorem}
\newtheorem{corollary}[corollary]{Corollary}
\aliascntresetthe{corollary}
\crefname{corollary}{corollary}{corollaries}
\Crefname{corollary}{Corollary}{Corollaries}

\newaliascnt{fact}{theorem}

\aliascntresetthe{fact}
\crefname{fact}{fact}{facts}
\Crefname{fact}{Fact}{Facts}

\newtheorem{definition}{Definition}[section]
\crefname{definition}{definition}{definitions}
\Crefname{definition}{Definition}{Definitions}

\newtheorem{assumption}{Assumption}[section]
\crefname{assumption}{assumption}{assumptions}
\Crefname{assumption}{Assumption}{Assumptions}

\crefname{example}{example}{examples}
\Crefname{example}{Example}{Examples}

\crefname{table}{table}{tables}
\Crefname{table}{Table}{Tables}

\crefname{figure}{figure}{figures}
\Crefname{figure}{Figure}{Figures}

\makeatletter
\newcommand{\newreptheorem}[2]{%
\newenvironment{rep#1}[1]{%
 \def\rep@title{#2~\ref{##1}}%
 \begin{rep@theorem}}%
 {\end{rep@theorem}}}
\makeatother

\newreptheorem{theorem}{Theorem}
\newreptheorem{lemma}{Lemma}

\crefname{reptheorem}{theorem}{theorems}
\Crefname{reptheorem}{Theorem}{Theorems}

\crefname{replemma}{lemma}{lemmas}
\Crefname{replemma}{Lemma}{Lemmas}

\let\top\intercal

\title{Acceleration for Affine-coupling problems}
\author{Clement Lezane, Sophie Langer, Wouter M Koolen}

\begin{document}

\maketitle

\begin{abstract}
Saddle-point problems play an important role in modern machine learning, including robust optimization and algorithmic fairness.
We investigate structured convex--concave minimax problems where the primal variable lies in a high-dimensional, unconstrained space $\mathbb{R}^d$, while the dual variable is confined to a constrained, low-dimensional space $\mathbb{R}^J$, with $J\ll d$.
Although existing theoretical frameworks establish that an accelerated rate of $\mathcal{O}(L/T^2)$ is possible, solving the associated subproblems efficiently can become a computational bottleneck in high dimensions.
To address this, we introduce a decoupling technique. By restricting second-order updates to the low-dimensional dual space and employing first-order methods in the high-dimensional primal space, our framework combines Newton methods with Nesterov acceleration.
For simplex-constrained dual variables and regularizers compatible with a self-concordant barrier, the resulting algorithm achieves a convergence rate of $\mathcal{O}(L/T^2)$ in primal objective suboptimality after $T$ outer iterations, without strong dual concavity.
The additional linear-algebra cost per outer iteration is $\widetilde{\mathcal{O}}(dJ^2+J^{3.5})$, alongside one evaluation of the component losses and their Jacobian.

\end{abstract}

\section{Introduction}

We study minimax optimization problems of the form
\begin{equation}
\label{eq:saddle_point}
\begin{aligned}
\min_{w\in\mathbb{R}^d}\max_{\lambda\in\Lambda}\Phi(\lambda,w),
    \qquad
    \Phi(\lambda,w)
    &:=\langle\lambda,f(w)\rangle-R(\lambda).
\end{aligned}
\end{equation}
Here $f:\mathbb{R}^d\to\mathbb{R}^J$ is smooth,
$\Lambda\subset\mathbb{R}^J$ is compact and convex, and
$R:\Lambda\to\mathbb{R}$ is convex. We call this structure \emph{affine
coupling}: the coupling term is linear in the typically low-dimensional
variable $\lambda$, whereas $f$ may depend nonlinearly on the
high-dimensional variable $w$. Under \Cref{Assumption-linear-convex}, $\Phi$ is
convex in $w$ and concave in $\lambda$.

Problems of this form arise naturally when several component losses must be
aggregated. Let $f_j(w)$ denote the loss associated with component $j$.
Choosing $\Lambda=\Delta_J$ and
$R(\lambda)=\frac{\gamma}{2}\|\lambda-\frac{1}{J}\mathbf{1}\|_2^2$ gives
\begin{align*}
    \min_{w\in\mathbb{R}^d}\max_{\lambda\in\Delta_J}
    \left\{
        \sum_{j=1}^J\lambda_j f_j(w)
        -\frac{\gamma}{2}
        \left\|\lambda-\frac{1}{J}\mathbf{1}\right\|_2^2
    \right\}.
\end{align*}
For $\gamma=0$, the inner maximum is $\max_j f_j(w)$; increasing $\gamma$
penalizes concentrated weights and moves the objective toward the uniform
average $J^{-1}\sum_j f_j(w)$. When $j$ indexes demographic groups or
classes, this yields objectives used in fair and class-imbalanced learning,
for example with $f_j(w)=\mathbb{E}[\ell(X,Y;w)\mid Y=j]$
\citep{hashimoto2018fairness,martinez2020minimax}. When $j$ indexes
environments or perturbation models, it yields a robust learning objective
\citep{sagawa2020distributionally,madry2018towards}.
Divergence-based and entropic regularizers yield related objectives in
distributionally robust and entropy-regularized learning
\citep{namkoong,hu2020kldro,ziebart2008maximum,haarnoja2018soft,
christodoulou2019soft,geist2019theory}.

From an optimization perspective, affine coupling lies between two familiar
regimes. Generic first-order methods for convex--concave saddle-point problems
have a worst-case convergence rate of order $1/T$
\citep{nesterov2018lectures,BT12}. When the dual variable disappears, for
example, when $\Lambda=\Delta_1=\{1\}$, the problem reduces to smooth convex
minimization, for which acceleration yields the optimal rate $1/T^2$
\citep{Nesterov83}. This raises our central question: can affine coupling recover the accelerated rate without strong concavity, while remaining computationally efficient when \(J\ll d\)?

To answer this question, we design an algorithm that achieves the optimal accelerated convergence rate of $\mathcal{O}(1/T^2)$ while maintaining a per-iteration complexity that scales gracefully to high-dimensional settings. Specifically, our contributions are threefold:

\textbf{1. Decoupled acceleration framework.} Assuming oracle access to $f$ and its Jacobian, we propose an algorithm achieving an $O(1/T^2)$ convergence rate for affine-coupling saddle-point problems \eqref{eq:saddle_point}. By leveraging a fixed-point structure to decouple the high-dimensional primal acceleration $w$ from the low-dimensional dual constraints $\lambda$, we restrict exact second-order updates exclusively to the dual space. This bypasses the traditional bottleneck in high dimensions, reducing the per-iteration linear-algebra complexity from $\mathcal{O}((d+J)^3)$ for a joint dense solve to strictly   $\mathcal{O}(dJ^2+J^3)$.

\textbf{2. Connecting interior-point methods and acceleration.} We relate the local Newton decrement used in classical interior-point methods \citep{nesterov2018lectures,boyd2004convex} to the global approximation error required by the accelerated outer method \citep{Nesterov83}. This theoretical bridge formally establishes that the nested inner problems can be solved only approximately while perfectly retaining the $O(1/T^2)$ convergence rate, translating theoretical composite acceleration into a computationally tractable algorithm.

\textbf{3. Extending interior-point guarantees for general geometries.} To rigorously bound the approximation error required by our first-order method, we extend the theoretical analysis of the classical path-following algorithm introduced by \citet{nesterov1994interior}. Specifically, we derive a novel analytical bound on the global cross-evaluation gap that holds for any compact convex constraint set $\Lambda$ equipped with a self-concordant barrier. 

\subsection{Literature review}

For general smooth convex-concave problems, optimistic and extragradient
methods \citep{Rakhlin13,mokhtari20a} and smoothing-based methods
\citep{BT12} achieve $\mathcal{O}(1/T)$ rates. These match worst-case
first-order lower bounds when the dimensions are sufficiently large relative
to the number of oracle calls \citep{Ouyang2018LowerCB}. Faster rates are
available under strong convexity or concavity \citep{Thekumprampil}, and
for the two-objective case $R=0$, $J=2$ \citep{thesis}.

Breaking the $\Omega(1/T)$ lower bound in bilinear games requires the dual update to utilize non-first-order information \citep{Ouyang2018LowerCB}. Composite optimization provide a potential pathway to acceleration for our problem
class. Indeed, \eqref{eq:saddle_point} can be written as
\[\min_w h(f(w)), \qquad
    h(z)=\max_{\lambda\in\Lambda}\{\langle\lambda,z\rangle-R(\lambda)\}.
\]
Following this perspective, \citet{drusvyatskiy2017} propose an accelerated prox-linear method that theoretically achieves an $\mathcal{O}(1/T^2)$ outer convergence rate under standard convexity conditions. However, their algorithm requires repeatedly solving $d$-dimensional regularized primal subproblems to high precision. When $f$ is nonlinear, finding these high-precision solutions is highly non-trivial. Employing general higher-order methods to solve these subproblems incurs prohibitive per-iteration costs, rendering the approach computationally intractable for large dimensions $d$. \Cref{table-literature} summarizes the theoretical convergence rates of existing algorithms relative to our approach, while \Cref{sec-complexity} provides a detailed accounting of both oracle and computational complexities.
\begin{table}[ht]
\centering
\label{tab:convergence}
\renewcommand{\arraystretch}{1.5}
\begin{tabular}{lll}
\toprule
\textbf{Setting} & \textbf{Existing method} & \textbf{Our method} \\
\midrule
\multirow{3}{*}{Affine-coupling}         & $\mathcal{O}(\frac{L}{T})$ \citet{Rakhlin13}    & \multirow{5}{*}{$\mathcal{O}\left( \frac{L}{T^2} \right)$} \\
          & $\mathcal{O}(\frac{L}{T})$ \citet{mokhtari20a} & \\
          & $\mathcal{O}(\frac{L}{T^2})^\dagger$ \citet{drusvyatskiy2017} & \\
$R$ $\gamma$-Strongly-Convex & $\Omega(\frac{L }{\gamma T^2})$ \citet{Ouyang2018LowerCB}  & \\
$R=0, J=2$&$\mathcal{O} (\frac{L\log(T)}{T^2})$ \citet{thesis}  & \\
\bottomrule
\end{tabular}
\caption{Convergence rate comparison of various algorithms given $T$ queries to $\nabla f$. $^\dagger$ whenever its subproblem oracle admits an efficient implementation.}
\label{table-literature}
\end{table}

\subsection{Preliminaries}
In this section, we introduce the preliminaries needed for our analysis. 
\begin{definition}[Bregman divergence]\label{def:bregman}
For a continuously differentiable function $F :\mathbb{R}^d  \xrightarrow{} \mathbb{R} $, the \emph{Bregman divergence} from $x \in \mathbb{R}^d$ to $y \in \mathbb{R}^d$ is defined by
\[ D_F(x,y) := F(x) - F(y) - \langle \nabla F(y),x-y \rangle. \]
\end{definition}
\begin{definition}[L-Smoothness and $\gamma$-Strongly-convexity]
\label{def:smooth}
A continuously differentiable function  $F :\mathbb{R}^d  \xrightarrow{} \mathbb{R} $ is said to be \emph{$L$-smooth}, if for all $x,y \in \mathbb{R}^d$
\[ D_F(x,y) \leq \frac{L}{2} \|x-y\|_2^{2}.\]
$F$ is said to be \emph{$\gamma$-strongly-convex}, if for all $x,y \in \mathbb{R}^d$
\[ D_F(x,y) \geq \frac{\gamma}{2} \|x-y\|_2^{2}.\]
\end{definition}
\begin{definition}[$1$-Self-concordance]\label{def:self_concordant}
A three times continuously differentiable convex function $F : \mathcal{D} \xrightarrow{} \mathbb{R}$, where $\mathcal{D} \subseteq \mathbb{R}^d$ is an open convex set, is said to be \emph{$1$-self-concordant} if for all $x \in \mathcal{D}$ and all directions $h \in \mathbb{R}^d$,
\[ 
    |\nabla^3 F(x)[h,h,h]| \leq 2\left( h^\top \nabla^2 F(x) h \right)^{3/2}, 
\]
where $\nabla^3 F(x)[h,h,h] := \left. \frac{d^3}{dt^3} F(x + th) \right|_{t=0}$ denotes the third-order directional derivative of $F$ at $x$ along the direction $h$.
\end{definition}

\begin{definition}[$\nu$-Self-concordant barrier]\label{def:nu_self_concordant_barrier}
A function $F : \mathcal{D} \to \mathbb{R}$, where $\mathcal{D} \subseteq \mathbb{R}^d$ is an open convex set, is said to be a \emph{$\nu$-self-concordant barrier} for its closure $\text{cl}(\mathcal{D})$ if it is a $1$-self-concordant function, $F(x) \to \infty$ as $x \to \partial\mathcal{D}$, and there exists a parameter $\nu \geq 0$ such that for all $x \in \mathcal{D}$ and all directions $h \in \mathbb{R}^d$,
\[ 
    |h^\top \nabla F(x)| \leq \sqrt{\nu} \left( h^\top \nabla^2 F(x) h \right)^{1/2}.
\]
Equivalently, when the Hessian $\nabla^2 F(x)$ is strictly positive definite, this condition can be written in terms of the Newton decrement as $\nabla F(x)^\top [\nabla^2 F(x)]^{-1} \nabla F(x) \leq \nu$.
\end{definition}

Throughout this work, the affine-coupling problem is analyzed under the following assumptions.
\begin{assumption}[Affine-Coupling]
\label{Assumption-linear-convex}
Consider the setting of \eqref{eq:saddle_point}, and let $L,G\geq0$.
\begin{itemize}
\item The set $\Lambda\subset\mathbb{R}^J$ is non-empty, compact, and convex.
The primal and dual dimensions satisfy $d,J\in\mathbb{N}_+$, and our main
computational regime of interest is $d\gg J$.
\item There exists a $\nu$-self-concordant barrier $B_\Lambda: \operatorname{relint}(\Lambda) \to \mathbb{R}$.
\item The function $R:\Lambda\to\mathbb{R}$ is convex and $R$ compatible with $B_{\Lambda}$ (for every $c\geq0$, $cR+B_{\Lambda}$ is convex and $1$-self-concordant.)
\item For every $\lambda\in\Lambda$, the function
$w\mapsto f(w)\lambda$ is convex and $L$-smooth, and its gradient is
bounded by $G$. 
\end{itemize}
\end{assumption}

\section{Acceleration with a Fixed-Point Oracle}
In this section we propose an algorithm and prove an accelerated convergence result for it. The construction uses an efficient first-order iterative scheme for $w$ and a (non-constructive) fixed-point oracle for $\lambda$.  This section serves as a warm up with clean analysis. Our main result, in the next section, tackles the challenge of efficiently implementing an approximate fixed point algorithm.

The following fixed-point property serves as the primary  motivation for our design:

\begin{theorem}
\label{thm-fixed-point}
Consider any function $\Phi:  \Lambda \times \mathbb{R}^d \xrightarrow{} \mathbb{R}$ that is concave and continuous in the first argument and $C^1$ in the second argument. Let $\Lambda \subset \mathbb{R}^J$ be a non-empty, convex and compact subset. Then for any continuous update rule
\begin{align*}
g_t : 
\begin{cases}
\Lambda \xrightarrow{} \mathbb{R}^d \\
\lambda\xrightarrow{} g_t(\lambda),\\
\end{cases}
\end{align*} there exists $\lambda_\star \in \Lambda$ such that
\[ \Phi(\lambda_\star,g_t(\lambda_\star)) = \max_{\lambda \in \Lambda}  \Phi (\lambda,g_t (\lambda_\star)). \]
\end{theorem}
The proof is provided in \Cref{appendix-proof-fixed-point}. This theorem provides a crucial structural guarantee: if we treat $g_t(\lambda)$ as a primal update step, there always exists a dual solution $\lambda_t^\star$ that is optimal with respect to the exact primal variables it induces. To make this concrete and integrate it into our framework, we instantiate $g_t$ as our primal optimization step. Specifically, for a fixed $\lambda_t^\star$, the gradient descent update rule is defined as:
\begin{align*}
g_t(\lambda) = w_{t+1}(\lambda) :
\begin{cases}
\Lambda \xrightarrow{} \mathbb{R}^d  \\
\lambda\xrightarrow{}  w_t - \eta_t  \nabla f(w_t) \lambda
\end{cases}.
\end{align*}
with $\eta_t \geq 0$ the stepsize at iteration $t$. By embedding this structural property directly into the standard Nesterov acceleration framework, we obtain the following update rules for the primal variable $w$::
\begin{equation}
\label{eq-perfect-acceleration}
\begin{aligned}
\begin{cases}
w_{t}^{md}  & = \frac{A_{t-1}}{A_{t}} w_{t}^{ag} + \frac{\alpha_{t}}{A_{t}} w_{t} \\
w_{t+1}  & =  w_t - \alpha_t \nabla f(w_t^{md}) \lambda_t^{\star}    \\
w_{t+1}^{ag}  & = \frac{A_{t-1}}{A_{t}} w_{t}^{ag} +\frac{\alpha_{t}}{A_{t}} w_{t+1}
\end{cases}
\end{aligned}
\end{equation}
The step-size parameters $\alpha_t$ and accumulation weights $A_t$ follow a slight variation of the classical optimal sequences established in the accelerated optimization literature \citep{nesterov2018lectures}. With the initial condition $A_0 = 0$, and assuming $L$ is the Lipschitz constant of the gradients defined later in the text, the parameter scheduling is governed by:
\begin{equation}
\label{eq-step-sizes}
\begin{aligned}
\begin{cases}
\forall t \geq 1, \quad  \alpha_t = \frac{1 + \sqrt{1 + 4L A_{t-1}}}{4L} \\
\forall t \geq 1, \quad  A_t = A_{t-1} + \alpha_t.\\
\end{cases}
\end{aligned}
\end{equation} 
Regarding the dual variable $\lambda$, we assume access to an oracle capable of computing the exact fixed point $\lambda_t^\star$ guaranteed by \Cref{thm-fixed-point} at each iteration $t$. This fixed point satisfies the following condition:
\begin{equation}
\label{eq-perfect-fixed-point}
{\lambda_t^{\star}}^\top f\left(w_t^{md} - \frac{\alpha_t^2}{A_t} \nabla f(w_t^{md}) \lambda_t^{\star}\right) - R(\lambda_t^{\star} ) = \max_{\lambda \in \Lambda} \left[ \lambda^\top f\left(w_t^{md} - \frac{\alpha_t^2}{A_t} \nabla f(w_t^{md}) \lambda_t^{\star}\right) - R(\lambda) \right],
\end{equation}
where we identify the new aggregated iterate as $w_{t+1}^{ag} = w_t^{md} - \frac{\alpha_t^2}{A_t} \nabla f(w_t^{md})\lambda_t^{\star}$. We present the main convergence theorem.

\begin{theorem}
Suppose that \Cref{Assumption-linear-convex} holds. Following the update rules at \Cref{eq-perfect-acceleration,eq-step-sizes,eq-perfect-fixed-point}, after $T$ iterations, for all $w_\star \in \mathbb{R}^d$
\[  \max_{\lambda \in \Lambda} \Phi(\lambda, w_{T+1}^{ag}) - \max_{\lambda \in \Lambda} \Phi(\lambda, w_{\star}) = \mathcal{O} \left(\frac{ L \| w_\star - w_{\mathrm{init}}\|_2^2 }{T^2} \right) \]
\end{theorem} 
The full proof is deferred to \Cref{appendix-proof-accelerated-noiseless}. In general, computing the exact fixed point $\lambda_t^{\star}$ at each step is computationally intractable. Consequently, in practical implementations, we solve an inner optimization problem to approximate this fixed point. We formalize this in the following sections, explicitly accounting for the resulting approximation errors in the final convergence rate

\section{Main results}
In this section we efficiently implement the fixed-point  oracle approximately. To do so, we  employ the Majorization-Minimization (MM) principle to formulate a sequence of surrogate subproblems, which we then solve using an interior-point method.

\subsection{Cross-evaluation gap}
 Following the standard Majorization-Minimization (MM) framework \citep{Lange,Rangarajan}, the natural surrogate $\Tilde{g}_t : \mathbb{R}^d \xrightarrow[]{} \mathbb{R}$ for minimizing an $L$-smooth function $g : \mathbb{R}^d \xrightarrow[]{} \mathbb{R}$ at iteration $t$ is given by:
\[\Tilde{g}_t(w) := g(w_t^{md}) + \langle \nabla g(w_t^{md}), w- w_t^{md} \rangle + \frac{L}{2} \Vert{} w- w_t^{md}\Vert{}_2^2,\]
where $w_t^{md}$ denotes the current iterate. The crucial step in extending this approach to saddle-point problems is parameterizing the update rule as $\lambda \xrightarrow{} w_{t+1}^{ag}(\lambda)$ and deriving a surrogate  $\Tilde{\mu} \Tilde{f}(w_{t+1}(\lambda),w_t^{md})$ for any $\mu \in \Lambda$ for all $\mu \in \Lambda$. Ultimately, our algorithm aims to find a $\lambda \in \Lambda$ that minimizes the gap between $\max_{\mu}\Phi(\mu, w_{t+1}^{ag}(\lambda))$ and $\Phi(\lambda, w_{t+1}^{ag}(\lambda))$.

To precisely quantify this difference, we introduce the cross-evaluation gap $\epsilon_t$ defined as:
\begin{equation}
\label{eq-moreau-env}
\begin{aligned} 
\forall \lambda \in \Lambda, \quad S_t(\lambda) & := \frac{\alpha_t^2}{2 A_t} \Vert{}\nabla f(w_t^{md}) \lambda\Vert{}_2^2 - \lambda^{\top} f(w_t^{md}) \\
\forall \lambda, \mu \in \Lambda, \quad \epsilon_t(\lambda,\mu) & := \langle \nabla S_t(\lambda), \lambda - \mu \rangle + R(\lambda) - R(\mu).
\end{aligned}
\end{equation}
The following theorem characterizes the behavior of the saddle-point surrogate within the affine-coupling problem and specifies the class of inner subproblems that must be solved at each iteration.
\begin{theorem}[Cross-evaluation gap]
\label{thm-gap}
Suppose that \Cref{Assumption-linear-convex} holds. We follow the update rules at \Cref{eq-perfect-acceleration}, the cross-evaluation gap is strictly bounded by:
\begin{align*}
\forall \lambda,\mu \in \Lambda, \quad \Phi(\mu, w_{t+1}^{ag}(\lambda)) - \Phi(\lambda, w_{t+1}^{ag}(\lambda)) \le  \epsilon_t(\lambda, \mu) + \frac{L }{2 } \| w_{t+1}^{ag}(\lambda) - w_t^{md} \|_2^2
\end{align*}
and by taking the maximum over $\mu$, for any arbitrary $\lambda_{t+1} \in \Lambda $:
\begin{align*}
& \max_{\mu \in \Lambda}\Phi(\mu, w_{t+1}^{ag}(\lambda_{t+1})) - \Phi(\lambda_{t+1}, w_{t+1}^{ag}(\lambda_{t+1})) \\
\le & \max_{\mu \in \Lambda} \epsilon_t(\lambda_{t+1},\mu) + \frac{L}{2 } \| w_{t+1}^{ag}(\lambda_{t+1}) - w_t^{md} \|_2^2.
\end{align*}
\end{theorem}
The subsequent theorem demonstrates that minimizing this surrogate directly yields the desired accelerated convergence rate.
\begin{theorem}[Accelerated algorithm]
Suppose that \Cref{Assumption-linear-convex} holds. We follow the update rules at \Cref{eq-perfect-acceleration,eq-step-sizes} with $\lambda_{t+1}$ verifying the follow property with $\epsilon_t \geq 0$:
\begin{equation}
\label{eq-epsilont-t}
\max_{\lambda \in \Lambda}  \Phi (\lambda,w_{t+1}^{ag}) - \Phi(\lambda_{t+1}, w_{t+1}^{ag}) \leq \epsilon_t + \frac{L}{2} \|w_{t+1}^{ag} - w_t^{md} \|_2^2
\end{equation}
Suppose that \Cref{eq-epsilont-t} holds for every iteration $t$. Then, after $T$ iterations, we know that for all $w_\star \in \mathbb{R}^d$
\[  \max_{\lambda \in \Lambda} \Phi(\lambda, w_{T+1}^{ag}) - \max_{\lambda \in \Lambda} \Phi(\lambda, w_{\star})= \mathcal{O} \left(\frac{ L \| w_\star - w_{\mathrm{init}}\|_2^2 }{T^2} +  \frac{1}{T^2} \sum_{t=1}^{T} t^2 \epsilon_t  \right) \]
\end{theorem}
We defer the detailed proof to \Cref{appendix-proof-accelerated-with-noise} and \Cref{appendix-proof-cross-gap}. Furthermore, the previous theorem highlights that accelerating convergence strictly relies on solving the inner problems to a precision of $\epsilon_t = \mathcal{O}(\frac{1}{t^4})$ .

\section{Interior-Point Method and cross-evaluation gap}
At each outer iteration $t$, we solve the following convex inner minimization subproblem:
\[\min_{\lambda \in \Lambda} h_t(\lambda) := S_t(\lambda) + R(\lambda).\]
where $S_t$ is a quadratic function define in \Cref{eq-moreau-env}. Since $S_t$ and $R$ are both assumed to be convex, the following inequality holds for all $\lambda,\mu \in \Lambda$ :
\begin{equation}\label{eq:chain}
(S_t + R)(\lambda) - (S_t + R)(\mu)  \leq  \langle \nabla S_t(\lambda), \lambda - \mu \rangle + R(\lambda) - R(\mu)   \le \langle \nabla (S_t+R)(\lambda), \lambda - \mu \rangle.
\end{equation}
Taking the maximum over $\mu \in \Lambda$ throughout this inequality naturally connects three fundamental quantities of the subproblem:
\begin{itemize}
    \item \textbf{Optimality gap}: $\max_{\mu \in \Lambda} \Big[ (S_t + R)(\lambda) - (S_t + R)(\mu) \Big] = h_t(\lambda) - \min h_t$
    \item \textbf{Cross-evaluation gap}: $\max_{\mu \in \Lambda} \Big[ \langle \nabla S_t(\lambda), \lambda - \mu \rangle + R(\lambda) - R(\mu) \Big]  = \max_{\mu \in \Lambda} \epsilon_t(\lambda,\mu)$
    \item \textbf{Frank-Wolfe gap}: $\max_{\mu \in \Lambda} \langle \nabla (S_t+R)(\lambda), \lambda - \mu \rangle = \max_{\mu \in \Lambda} \langle \nabla h_t(\lambda),\lambda - \mu \rangle$
\end{itemize}
The chain of inequalities in \Cref{eq:chain} demonstrates that the optimality gap is upper-bounded by the cross-evaluation gap, which is itself bounded by the Frank-Wolfe  gap. When $R$ is self-concordant, the problem of minimising $h_t$ is well-studied, allowing us to leverage standard algorithms and established convergence guarantees for the Frank-Wolfe gap. In such cases, these results can be directly applied to our cross-evaluation gap to ensure general convergence (see \Cref{sec-inner-solve-simple} for an example when $R$ is quadratic).

However, in the general setting, $R$ may be \textit{compatible} (i.e., $cR+B$ is self-concordant) rather than strictly self-concordant. This is the case, for example, when $R(\lambda):= \sum_{j=1}^J \lambda_j \log (\lambda_j)$ is the negative Shannon entropy and $B_{\Lambda} = -\sum_{j=1}^J \log(\lambda_j)$ is the negative Burg entropy (See \Cref{lemma-entropy}).

To the best of our knowledge, the Frank-Wolfe gap evaluated at the output of standard interior-point solvers is not known to be tightly bounded. For instance, the path-following algorithm of \citet{nesterov1994interior}, which optimizes the barrier-penalized objective  $h_t + \frac{B_{\Lambda}}{\tau}$ , returns a solution $\lambda \in \Lambda$ satisfying 
\[ h_t(\lambda) - \min h_t \leq \frac{\nu}{\tau},\] but provides no theoretical guarantee for the Frank-Wolfe gap itself. Therefore, to extract a valid upper bound for the cross-evaluation gap $\epsilon_t$, we must delve into the underlying mechanics of the path-following framework \citep{nesterov1994interior} and explicitly exploit the quadratic structure of $S_t$.

In the remainder of this section, we first provide a concrete example of adapting an existing solver for our subproblem when $R$ is quadratic (\Cref{sec-inner-solve-simple}). We then establish theoretical guarantees for the more challenging general case of compatible regularizers in \Cref{sec-inner-solve-general}.

\subsection{Standard Geometries and Optimality Certificates}
\label{sec-inner-solve-simple}

In this section, we demonstrate how to integrate highly optimized, off-the-shelf solvers to address specific affine-coupling settings, culminating in a practical, end-to-end algorithmic solution. Specifically, when $R$ is quadratic and $\Lambda$ is a simplex, the inner subproblem reduces to a standard quadratic program. We can therefore delegate it to solvers as SciPy's \texttt{trust-constr}. Such solvers typically measure convergence via the inexact Karush-Kuhn-Tucker (KKT) optimality conditions of the constrained problem. To bridge this practical solver output with our theoretical guarantees, the following lemma translates these KKT tolerances into a formal bound on the Frank-Wolfe gap, thereby establishing an exact stopping criterion for our algorithm.
\begin{lemma}
\label{lem-kkt-pure-quadratic}
Suppose an interior-point algorithm terminates at a strictly feasible output $\mu_k \in\text{relint}(\Delta_J)$, returning a barrier parameter $1 /\tau_k \le \epsilon_\tau$ and dual multipliers $\nu_k \in \mathbb{R}$ and $z_k \in \mathbb{R}_{+}^{J}$ that satisfy the following inexact Karush-Kuhn-Tucker (KKT) conditions
\begin{itemize}
    \item \textbf{Approximate Stationarity} $\vert{}\vert{}\nabla h_{t}(\mu_{k})-\nu_{k}1-z_{k}\vert{}\vert{}_{\infty}\le\epsilon_{g} $
    \item \textbf{Perturbed Complementary Slackness} $\mu_{k,j} z_{k,j} = 1/\tau_{k}$ for all $j\in\{1,...,J\}$
\end{itemize}
where $\epsilon_{g}$ is the tolerance bar for the stationary condition. Then, the Frank-Wolfe gap at $\mu_k$ is strictly bounded by:
\begin{align*}
\max_{\mu \in \Delta_J} \langle \nabla h_t(\mu_k), \mu_k - \mu \rangle \le J\epsilon_{\tau}+2\epsilon_{g}.
\end{align*}
\end{lemma}

We defer the proof of this lemma to \Cref{sec-kkt-pure-quadratic}. To ensure the Frank-Wolfe duality gap remains below the target precision $\frac{L}{t^4}$, we set the tolerances to $\epsilon_{\tau} := \frac{L}{2Jt^4}$ and $\epsilon_g := \frac{L}{4t^4}$. Given these choices, \citet{boyd2004convex} establish that a standard quadratic solver satisfies the stopping criteria in $\mathcal{O} \left( \sqrt{J} \log \left( \frac{1}{\epsilon_\tau} \right) \right)$ iterations. Consequently, \Cref{alg:accelerated_method_quadratic} yields an end-to-end algorithm achieving an $\mathcal{O}(1/T^2)$ accelerated convergence rate for the affine-coupling problem with a quadratic regularizer $R$ over the simplex $\Lambda = \Delta_J$.

\begin{algorithm}
\caption{Accelerated Gradient Method with Quadratic surrogate}
\label{alg:accelerated_method_quadratic}
\begin{algorithmic}[1]
\State \textbf{Initialize:} Choose $w_1   \in \mathbb{R}^d$. Set $L$ and $T$. 
\State Set $A_0 =0, w_1^{md} = w_1 = w_1^{ag} $
\For{$1\leq t \leq T$}
    \State Compute $\alpha_t = \frac{1 + \sqrt{1 + 4L A_{t-1}}}{4L}, A_t = A_{t-1} + \alpha_t $
    \State Compute $w_{t}^{md} = \frac{A_{t-1}}{A_{t}} w_{t}^{ag} + \frac{\alpha_{t}}{A_{t}} w_{t}$
    \State Compute the Jacobian $\nabla f(w_t^{md})$ and the vector $f(w_t^{md})$
    \State Compute the surrogate $h_t(\lambda) := \frac{\alpha_t^2}{2 A_t} \|\nabla f(w_t^{md}) \lambda \|_2^2 - \lambda^{\top} f(w_t^{md}) + R(\lambda)$
    \State Call QP solver to minimize $h_t$ in the simplex with tolerance rates $\left( \frac{L}{2Jt^4}, \frac{L}{4t^4} \right)$
    \State \qquad Obtain $\lambda_t^{\star} \in \Delta_J $
    \State Update $w_{t+1} = w_t - \alpha_t \nabla f(w_t^{md}) \lambda_t^{\star}$
    \State Update $w_{t+1}^{ag} = \frac{A_{t-1}}{A_{t}} w_{t}^{ag} +\frac{\alpha_{t}}{A_{t}} w_{t+1}$
\EndFor
\State \Return $w_{T+1}^{ag}$
\end{algorithmic}
\end{algorithm}

\subsection{Custom Path-Following Method for General Barriers}
\label{sec-inner-solve-general}
In this section, we analyze the underlying mechanics of the interior point solver to derive a strict upper bound for the cross-evaluation gap, $\epsilon_t$. To formally establish convergence guarantees for the interior point method, we introduce the following condition on the optimal solution:
\begin{assumption}\label{ass:init}
Let $w_\star \in \mathbb{R}^d$ be the global minimum. We assume that $\forall 1 \le i, j \le J$, \[ |f_i(w_\star) - f_j(w_\star)| \le S_\star.\]
\end{assumption}
Following standard practices in interior point methods \citep{nesterov1994interior}, we consider the barrier-penalized objective parameterized by $\tau \ge 1$:
\begin{equation}
F_\tau(\lambda) := \tau(S_t(\lambda) + R(\lambda)) + B_\Lambda(\lambda).
\end{equation}
Under our working assumptions, $F_\tau$ is strictly convex and $1$-self-concordant. Minimizing this objective naturally yields the following Newton system, which we subsequently use to bound the cross-evaluation gap $\epsilon_t$.
\begin{definition}[Newton System and Decrement]\label{def-KKT}
Consider $F_\tau$ twice-continuously differentiable and a strictly feasible point $\mu \in \mathrm{relint}(\Lambda)$. Let $\{ A\lambda+ b\}$ be the Affine Hull of $\Lambda$. The Newton step direction $\Delta\mu \in \mathbb{R}^J$ and its associated local equality multiplier $s \in \mathbb{R}$ are defined as the solution to the symmetric linear system:
\begin{equation}
\begin{bmatrix} 
\nabla^2 F_\tau(\mu) & A \\ 
A & 0 
\end{bmatrix} 
\begin{bmatrix} 
\Delta\mu \\ 
s 
\end{bmatrix} 
= 
\begin{bmatrix} 
-\nabla F_\tau(\mu) \\ 
0 
\end{bmatrix}
\end{equation}
When a solution exists, we quantify the proximity to the exact central path with Newton decrement:
\begin{equation}
\delta(\mu, F_\tau) := \sqrt{\Delta\mu^\top \nabla^2 F_\tau(\mu) \Delta\mu}.
\end{equation}
\end{definition}
\begin{theorem}[Inner problem Upper bound via KKT block]
\label{thm-inner-problem-core}
Under \Cref{Assumption-linear-convex}, we consider a fixed $\tau \geq 1 $ and  $\mu_k \in  \text{relint}(\Lambda)$. Suppose that the KKT block defined on $(\mu_k, F_\tau)$ according to \Cref{def-KKT} has a feasible solution $(\Delta \mu_k, s_k)$ and there exists $ \beta \in [0,1)$,
\begin{align*}
\delta(\mu_k,F_{\tau})\leq \beta^2,
\end{align*}
then for all $\lambda \in \Lambda $
\begin{equation}
\epsilon_t(\mu_k,\lambda) \leq 
\frac{\nu}{\tau} + \frac{\sqrt{\nu}}{\tau} \left(\frac{\beta}{1-\beta}\right) + \frac{\omega_{\star}(\beta)}{\tau} + \frac{\beta}{1-\beta} \frac{2G\alpha_t}{\sqrt{\tau A_t}}
\end{equation}
where $\omega_\star(\beta) := -\beta - \ln(1-\beta)$.
\end{theorem}
The preceding theorem connects standard interior point theory with the cross-evaluation bound essential to our acceleration framework. In \Cref{appendix-simplex}, we detail \Cref{alg:two_phase_pfm}, an adaptation of the path-following algorithm from \citep{nesterov1994interior} tailored specifically to the simplex domain. The following corollary establishes a strict upper bound on the total number of iterations required by our inner loop algorithm.

\begin{corollary}
Assume $\Lambda = \Delta_J$ and $B_{\Lambda}$ to be negative Burg entropy. Suppose $D=\| w_{\star} - w_1\|_2,L \geq 1$. Let the parameters be set according to \Cref{eq-perfect-acceleration,eq-step-sizes} with $\beta=1/4$ and  $\epsilon_t :=L/t^4$. Under \Cref{Assumption-linear-convex,ass:init}, the total number of inner loop iterations $N_{\text{Inner}}$ in \Cref{alg:two_phase_pfm} satisfies
\begin{align*}
N_{\text{Inner}} = \mathcal{O} \left( G^2 + \sqrt{J} \log \Big( J G t \Big) +   \frac{G^2 }{L} +  2 G (\sqrt{D^2 + \log(t)} + S_\star  \right)
\end{align*}
\end{corollary}
The proof and further theoretical details are deferred to \Cref{proof-kkt-upper-bound}.

\section{Implementation}
\label{sec-complexity}
Before entering the simulation, we estimate the overall computational complexity of our algorithm. During the inner loop optimization over $\lambda$, the Jacobian of $f$ evaluated at $w_t^{md}$ remains fixed. Consequently, each inner loop requires exactly one evaluation of $\nabla f$ and at most $N_{\text{inner}}$ evaluations of the penalty derivatives $(\nabla R, \nabla^2 R)$. Unlike standard first-order methods, our second-order inner loop incurs the overhead of computing and inverting a $J \times J$ Hessian, adding $\mathcal{O}(dJ^2 + J^3 N_{\text{inner}})$ to the per-iteration cost. Factoring in the total iterations required to reach an $\epsilon$-approximate saddle point, we summarize the overall complexities in \Cref{tab:complexity} for a rigorous comparison with existing baselines.

\begin{table}[h]
\centering
\begin{tabular}{lccc}
\toprule
\textbf{Algorithms} & \textbf{Call} $\nabla f$ & \textbf{Call} $(\nabla R, \nabla^2 R)$ & \textbf{Computational cost} \\
\midrule
\citet{Rakhlin13} & $\mathcal{O}\left(\frac{L}{\epsilon}\right)$ & $\mathcal{O}\left(\frac{L}{\epsilon}\right)$ & $\mathcal{O}\left(\frac{L}{\epsilon} dJ\right)$ \\
\citet{Ouyang2018LowerCB} & $\mathcal{O}\left(\sqrt{\frac{L}{\gamma \epsilon}}\right)$ & $\mathcal{O}\left(\sqrt{\frac{L}{\gamma \epsilon}}\right)$ & $\mathcal{O}\left(\sqrt{\frac{L}{\gamma \epsilon}} dJ\right)$ \\
\textbf{Our methods} & $\mathcal{O}\left(\sqrt{\frac{L}{\epsilon}}\right)$ & $\tilde{\mathcal{O}}\left(\sqrt{\frac{L}{\epsilon}}\right)$ & $\tilde{\mathcal{O}}\left(\sqrt{\frac{L}{\epsilon}} \left(dJ^2 + J^{3.5}\right)\right)$ \\
\bottomrule
\end{tabular}
\caption{Overall comparison of various algorithms for the saddle point problem to reach precision $\epsilon$. The $\tilde{\mathcal{O}}$ notation hides additional multiplicative logarithmic factors.}
\label{tab:complexity}
\end{table}

\textbf{Synthetic Logistic Regression} To empirically validate our theoretical guarantees, we evaluate the proposed algorithm on a synthetic group-structured classification task. Specifically, we consider a regularized minimax logistic regression problem. Suppose we have $J$ demographic groups, where each group $j \in \{1, \dots, J\}$ contains a dataset of $N = 5,000$ samples $(X_{j,i}, y_{j,i}) \in \mathbb{R}^d \times \{-1, 1\}$. The component loss $f_j(w)$ is defined as the empirical logistic risk for group $j$:
\begin{align*}
f_j(w) = \frac{1}{N} \sum_{i=1}^{N} \log\left(1 + \exp\left(-y_{j,i} X_{j,i}^\top w\right)\right) \\
\min_{w \in \mathbb{R}^d} \max_{\lambda \in \Delta_J} \left\{ \sum_{j=1}^J \lambda_j f_j(w) - \frac{\gamma}{2} \left\Vert{}\lambda - \frac{1}{J}\mathbf{1}\right\Vert{}_2^2 \right\}.
\end{align*}
where $\Delta_J$ is the simplex and $\gamma \ge 0$ dictates the strong concavity of the regularizer. Our primary evaluation metrics are the iteration count and the GPU wall-clock time required to achieve an $\epsilon$-suboptimal primal gap (full simulation details are provided in \Cref{sec-appendix-simu}). We benchmark our approach against established first-order methods—Accelerated Proximal Gradient Descent (APGD) and Extragradient (EG)—as well as higher-order solvers like L-BFGS-B. The evaluation spans three main axes, with overall results summarized in \Cref{tab:all_experiments}:

\textbf{1. Robustness to the regularization parameter ($\gamma$).} A key theoretical advantage of our method is its convergence guarantees even in the absence of strong concavity ($\gamma = 0$). Empirically, while APGD degrades significantly for small values of $\gamma$, our method and EG remain robust.
    
\textbf{2. Accelerated high-precision convergence.} Compared to standard first-order approaches, the primary benefit of our accelerated algorithm is its ability to reach high-precision solutions in a fraction of the time.
    
\textbf{3. Scaling with dimensions ($d$ and $J$).} Although our higher-order implementation lacks low-level optimization, we compare its computational efficiency against SciPy's L-BFGS-B. Our algorithm exhibits superior scaling with respect to the primal dimension $d$, which offsets its theoretical complexity disadvantage in the dual dimension $J$.

\begin{table}[htbp]
\centering
\label{tab:all_experiments}

\vspace{0.2cm}
\textbf{(a) Convergence behavior of our algorithm (green) compared to APGD (orange) and EG (black). The left pane plots the primal gap against the iteration count, while the right pane plots it against computational wall-clock time.} ($d=5000, J=10$) \\
\vspace{0.1cm}
\includegraphics[width=\linewidth]{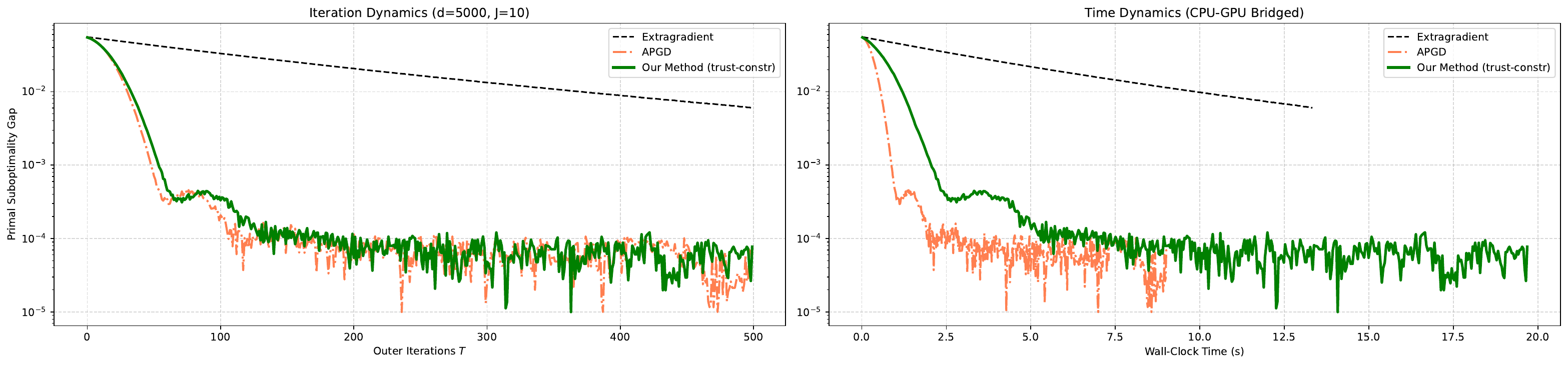}

\vspace{0.6cm} 

\textbf{(b) Time to reach precision $\epsilon=10^{-2}$} ($d=5000, J=10$) \\
\vspace{0.1cm}
\begin{tabular}{lccccc}
\toprule
\textbf{Algorithm} & $\gamma=0$ & $\gamma=0.001$ & $\gamma=0.01$ & $\gamma=0.1$ & $\gamma=1$ \\
\midrule
\textbf{Ours} & 34 (1.16s) & 35 (1.46s) & 34 (1.22s) & 34 (1.24s) & 34 (1.20s) \\
\textbf{APGD} & N/A & 88 (1.59s) & 32 (0.59s) & 20 (0.36s) & 18 (0.32s) \\
\textbf{EG}   & 210 (5.48s) & 219 (5.92s) & 372 (9.90s) & -/- & -/- \\
\bottomrule
\end{tabular}

\vspace{0.6cm} 

\textbf{(c) Time to reach precision $\epsilon$} ($d=5000, J=30$) \\
\vspace{0.1cm}
\begin{tabular}{lccc}
\toprule
\textbf{Algorithm} & $\epsilon=10^{-1}$ & $\epsilon=10^{-2}$ & $\epsilon=10^{-3}$ \\
\midrule
  \textbf{Ours} & 1 (0.02s) & 34 (1.42s) & 55 (2.38s) \\
  \textbf{APGD} & 1 (0.02s) & 32 (0.57s) & 49 (0.88s) \\
  \textbf{EG}   & 1 (0.03s) & 372 (9.88s) & -/- \\
\bottomrule
\end{tabular}

\vspace{0.6cm} 

\textbf{(d) Complexity scaling in $d$} ($J=10, \gamma =10^{-2}, T=200$) \\
\vspace{0.1cm}
\begin{tabular}{lcccc}
\toprule
& \multicolumn{2}{c}{\textbf{Ours}} & \multicolumn{2}{c}{\textbf{L-BFGS-B}} \\
\cmidrule(lr){2-3} \cmidrule(lr){4-5}
\textbf{Dimension $d$} & Obj. & Time (s) & Obj. & Time (s) \\
\midrule
$1000$ & 0.675157 & 5.64 & 0.675161 & 0.13 \\
$2000$ & 0.670539 & 6.74 & 0.670589 & 0.42 \\
$5000$ & 0.637399 & 7.71 & 0.637422 & 1.07 \\
\bottomrule
\end{tabular}
\vspace{0.6cm} 

\textbf{(e) Complexity scaling in $J$} ($d=5000, \gamma =10^{-2},T=200$) \\
\vspace{0.1cm}
\begin{tabular}{lcccc}
\toprule
& \multicolumn{2}{c}{\textbf{Ours}} & \multicolumn{2}{c}{\textbf{L-BFGS-B}} \\
\cmidrule(lr){2-3} \cmidrule(lr){4-5}
\textbf{Dimension $J$} & Obj. & Time (s) & Obj. & Time (s) \\
\midrule
$4$ & 0.545715 & 5.28 & 0.545861 & 0.36 \\
$8$ & 0.622306 & 6.50 & 0.622308 & 0.63 \\
$16$ & 0.658493 & 9.81 & 0.658977 & 1.67 \\
$32$ & 0.676453 & 16.84 & 0.676627 & 1.08 \\
\bottomrule
\end{tabular}
\caption{Summary of experimental results across all settings.}
\end{table}
\newpage
\subsection*{AI use statement}
In this work, we used generative AI tools to help develop to provide critical ingredients for proving mathematical claims, to assist in the writing of proofs, generate synthetic data sets and to assist with translation. 

Additionally, we used generative AI tools to transcribe recordings of research material, summarize or analyse existing literature, edit a research paper to improve readability. We have reviewed all AI-assisted work. LLM-generated code was verified and tested for correctness. We take responsibility for the final content of this work, including text, claims or artifacts produced with the aid of generative AI.

\subsection*{Reproducibility statement}
To ensure the full reproducibility of our work, we provide an in-depth breakdown of the synthetic data generation and experimental configurations. On the theoretical side, rigorous and complete proofs supporting all claims, theorems, and bounds are detailed in the Appendix.

\bibliography{ref}
\bibliographystyle{iclr2027_conference}

\appendix

\section{Complementary materials}
\begin{lemma}  \label{lem:prox}
Let $f$ be a convex function and $\nu$ be convew and continuously differentiable. For \[ u^{\star} = \arg\min_{u\in {\cal X}} \{f(u) + D^{\nu}(u,y)\}, \]
and all $u$, it holds
\[f(u^{\star}) + D^{\nu}(u^\star,y) + D^{f}(u,u^\star) \leq f(u) +  D^{\nu}(u,y) - D^{\nu}(u,u^\star). \]
\end{lemma}

\begin{proof}
As $u^\star$ is a minimizer, it fulfills for all  $u\in{\cal X}$ the first-order optimality condition
\begin{align}
\langle \nabla f(u^\star) + \nabla D^{\nu}(u^{\star},y) , u - u^{\star}\rangle \geq 0,
\end{align}
 with the gradient taken with respect to the first entry. Applying the three-points identity \cite{CT93}, yields 
\[\langle \nabla D^{\nu}(u^{\star},y) , u - u^{\star}\rangle = D^{\nu}(u,y)- D^{\nu}(u,u^{\star})- D^{\nu}(u^{\star},y),\]
leading to
\begin{align*}
  f(u) - f(u^{\star}) - D^{f}(u,u^{\star}) & = \langle \nabla f(u^\star) , u - u^{\star}\rangle \geq D^{\nu}(u,u^{\star})- D^{\nu}(u,y)  +D^{\nu}(u^{\star},y).
\end{align*}
\end{proof}

\begin{lemma}[Shannon Entropy]
\label{lemma-entropy}
If $R$ is the negative Shannon entropy and $B$ is the negative Burg entropy, for all $c \geq 0$, $cR+B$ is $1$-self-concordant. 
\end{lemma}
\begin{proof}
Let $c \ge 0$, the sum is:
\[g(\mu) = c \mu \ln \mu - \ln \mu\]
We calculate the derivatives 
\begin{align*}
g''(\mu) & = \frac{c}{\mu} + \frac{1}{\mu^2} = \frac{c\mu + 1}{\mu^2} \\
g'''(\mu) & = -\frac{c}{\mu^2} - \frac{2}{\mu^3} = -\frac{c\mu + 2}{\mu^3}
\end{align*}
Now, the $1$-self-concordance ($\vert{}g'''(\mu)\vert{} \le 2(g''(\mu))^{3/2}$) can be rewritten as :
\[\frac{c\mu + 2}{\mu^3} \le 2 \left( \frac{c\mu + 1}{\mu^2} \right)^{3/2} = \frac{2(c\mu + 1)^{3/2}}{\mu^3}\]
Since $\mu > 0$, we consider $y := c\mu \ge 0$ and multiply both sides by $\mu^3$
\begin{align*}
c\mu + 2 & \le 2(c\mu + 1)^{3/2} \\
\Longleftrightarrow \quad y + 2 & \le 2(y + 1)^{3/2}.
\end{align*}The inequality holds for $y=0$. The left side grows at a constant rate of $1$. The right side grows at a rate of $3(y+1)^{1/2} \ge 3$. Because the right side grows faster than the left side, the inequality is strictly and globally true for all $y \ge 0$.
\end{proof}

\section{Step-sizes}
\label{appendix-step-size}
We recall that the Nesterov accelerated gradient scheme:
\begin{equation*}
\begin{aligned}
\begin{cases}
 A_0 = 0 \\
\forall t \geq 1, \quad  \alpha_t = \frac{1 + \sqrt{1 + 4L A_{t-1}}}{4L} \\
\forall t \geq 1, \quad  A_t = A_{t-1} + \alpha_t \\
\end{cases}
\end{aligned}
\end{equation*}

\begin{lemma}
\label{lem-step-size}
We have for all $t\geq 1$
\begin{align*}
2 L \alpha_t^2 &  \leq  A_t  \\
\frac{t^2 + 4t}{16L} <  A_t & \leq \frac{(t+2)^2}{16L} + \frac{(t+2)\ln(t+1)}{4L} + \frac{\ln^2(t+1)}{4L}
\end{align*}
\end{lemma}
\begin{proof}
By the definition of $\alpha_t$,
\begin{align*}
4L \alpha_t - 1 & = \sqrt{1 + 4LA_{t-1}} \\
\implies (4L \alpha_t - 1)^2 & = 1 + 4LA_{t-1} \\
\implies 16L^2 \alpha_t^2 - 8L \alpha_t + 1 & = 1 + 4LA_{t-1}.
\end{align*}
After rearranging the terms, we have:
\begin{align*}
4L \alpha_t^2 - 2\alpha_t & = A_{t-1} = A_t - \alpha_t \\
\implies 4L \alpha_t^2 & = A_t + \alpha_t.
\end{align*}
\textbf{For the first inequality}, we obtain
\begin{align*}
\frac{L\alpha_t^2}{A_t} = \frac{A_t + \alpha_t}{4A_t} = \frac{1}{4} \left( 1 + \frac{\alpha_t}{A_t} \right) \le \frac{1}{4} (1 + 1) = \frac{2}{4} = \frac{1}{2}
\end{align*}
which yields the first inequality. Now continue to develop
\begin{align*}
4L \alpha_t^2 - 2\alpha_t & = A_{t-1} \\
\implies 4L^2 \alpha_t^2 - 2L\alpha_t + \frac{1}{4} & = LA_{t-1} + \frac{1}{4} \\
\implies \left( 2L \alpha_t - \frac{1}{2} \right)^2 & = LA_{t-1} + \frac{1}{4}
\end{align*}
Now, we take the square root and rearrange the terms
\begin{align*}
2L \alpha_t - \frac{1}{2} & = \sqrt{LA_{t-1} + \frac{1}{4}} \\
\implies L \alpha_t & = \frac{1}{4} + \frac{1}{2}\sqrt{LA_{t-1} + \frac{1}{4}} \\
\implies LA_t + \frac{1}{4} & = \left( LA_{t-1} + \frac{1}{4} \right) + \frac{1}{2}\sqrt{LA_{t-1} + \frac{1}{4}} + \frac{1}{4}
\end{align*}
Let us introduce $x_t = \sqrt{LA_t + \frac{1}{4}}$ with $x_0  = \frac{1}{2}$, substituting this yields :
\begin{align*}
x_t^2 = x_{t-1}^2 + \frac{1}{2}x_{t-1} + \frac{1}{4} 
\end{align*}
\textbf{For lower bound}, we notice that  for $t\geq 1$:
\begin{align*}
\left(x_{t-1} + \frac{1}{4}\right)^2 & = x_{t-1}^2 + \frac{1}{2}x_{t-1} + \frac{1}{16} \\
\implies x_t^2 & > \left(x_{t-1} + \frac{1}{4}\right)^2 \\
\implies x_t & > x_{t-1} + \frac{1}{4}
\end{align*}
By induction,
\begin{align*}
x_t > x_0 + \frac{t}{4} = \frac{t}{4} + \frac{1}{2} \\
x_t^2 > \frac{t^2}{16} + \frac{t}{4} + \frac{1}{4} \\
A_t > \frac{t^2 + 4t}{16L}
\end{align*}

\textbf{For upper bound} Since $x_t$ is increasing, we notice for $t\geq 1$
\begin{align*}
x_t - x_{t-1} = \frac{x_t^2 - x_{t-1}^2}{x_t + x_{t-1}} = \frac{\frac{1}{2}x_{t-1} + \frac{1}{4}}{x_t + x_{t-1}} < \frac{\frac{1}{2}x_{t-1} + \frac{1}{4}}{2x_{t-1}} = \frac{1}{4} + \frac{1}{8x_{t-1}} 
\end{align*}
We recall that $x_{t-1} \geq \frac{t-1}{4} + \frac{1}{2} = \frac{t+1}{4}$
\begin{align*}
x_t - x_{t-1} & \leq  \frac{1}{4} + \frac{1}{8\left(\frac{t+1}{4}\right)} = \frac{1}{4} + \frac{1}{2(t+1)} \\
\implies x_t - x_0 & \leq  \frac{t}{4} + \frac{1}{2} \sum_{k=1}^t \frac{1}{k+1}
\end{align*}
Since $\left(\sum_{k=1}^t \frac{1}{k+1} < \int_1^{t+1} \frac{du}{u} = \ln(t+1) \right)$, we get:
\begin{align*}
x_t & \leq \frac{t}{4} + \frac{1}{2} + \frac{1}{2}\ln(t+1) \\
x_t^2 & \leq \frac{(t+2)^2}{16} + \frac{(t+2)\ln(t+1)}{4} + \frac{\ln^2(t+1)}{4}\\
A_t & = \frac{x_t^2 - \frac{1}{4}}{L} \leq \frac{(t+2)^2}{16L} + \frac{(t+2)\ln(t+1)}{4L} + \frac{\ln^2(t+1)}{4L} .
\end{align*}
\end{proof}

\section{Fixed point Theorem}
\label{appendix-proof-fixed-point}

\begin{theorem}[Berge's Maximum Theorem, \cite{berge1997topological} ]
\label{thm-berge}
Let $X$ and $Y$ be topological spaces. Let $f: X \times Y \to \mathbb{R}$ be a continuous function, and let $C: Y \to 2^X$ be a continuous set-valued mapping (i.e. upper and lower hemicontinuous) with non-empty compact values. Then, 
\begin{itemize}
    \item The marginal maximum value function $M: Y \to \mathbb{R}$ defined by $M(y) = \max_{x \in C(y)} f(x, y)$ is continuous.
    \item The argmax correspondence $A: Y \to 2^X$ defined by $A(y) = \{x \in C(y) \mid f(x, y) = M(y)\}$ is upper hemicontinuous with non-empty and compact values.
\end{itemize}
\end{theorem}

\begin{theorem}[Kakutani's Fixed-Point Theorem, \cite{Kakutani1941AGO}]
\label{thm-kakutani}
Let $S$ be a non-empty, compact, and convex subset of a Euclidean space (or a locally convex topological vector space) $\mathbb{R}^n$. Let $\phi: S \to 2^S$ be a set-valued mapping. If $\phi$ is upper hemicontinuous (which implies a closed graph when $S$ is compact) and $\phi(x)$ is non-empty and convex for all $x \in S$, then $\phi$ has a fixed point. That is, there exists an $x^\star \in S$ such that $x^\star \in \phi(x^\star)$.
\end{theorem}

From the two previous Theorem, we can prove our Theorem.
\begin{theorem}
Suppose that $\Phi:  \Lambda \times \mathbb{R}^d \xrightarrow{} \mathbb{R}$ is concave in the first argument and $\mathcal{C}_1$ in the second argument. If $\Lambda \subset \mathbb{R}^J$ is a non empty, convex and compact subset, then for any continuous update rule
\begin{align*}
g_t : 
\begin{cases}
\Lambda \xrightarrow{} W \\
\lambda\xrightarrow{} g_t(\lambda)\\
\end{cases}
\end{align*} there exists $\lambda_\star \in \Lambda$ such that
\[ \Phi(\lambda_\star,g_t(\lambda_\star)) = \max_{\lambda \in \Lambda}  \Phi (\lambda,g_t (\lambda_\star)). \]
\end{theorem}
\begin{proof}
We consider the best-response mapping 
\begin{align*}
\forall \lambda \in \Lambda, \quad S(\lambda) := \arg\max_{\mu \in \Lambda} \Phi(\mu, g_t(\lambda))
\end{align*}
Because the objective $\Phi(\lambda; w_t)$ and the continuous update rule $g_t(\lambda)$ are continuous, and the constraint set $\Lambda$ is compact, \Cref{thm-berge} guarantees that for all $\lambda \in \Lambda$, $S(\lambda)$ is non-empty, compact-valued, and upper hemicontinuous (possesses a closed graph).

We note that the mapping $S(\lambda)$ is convex because the function $\Phi(\cdot, w_t)$ is concave. Since $S(\lambda)$ is an upper hemicontinuous, convex-valued mapping from a compact, convex set to itself, \Cref{thm-kakutani} guarantees it must have a fixed point $\lambda_\star \in S(\lambda_\star)$.
\end{proof}

\section{Proof for the Acceleration}
\subsection{Noiseless case}
\label{appendix-proof-accelerated-noiseless}
We recall  the accelerated algorithm 
\begin{equation}
\begin{aligned}
\begin{cases}
w_{t}^{md} & = \frac{A_{t-1}}{A_{t}} w_{t}^{ag} + \frac{\alpha_{t}}{A_{t}} w_{t} \\
& \lambda_t^{\star}  \quad  \textrm{Fixed point from \Cref{thm-fixed-point}} \\
w_{t+1} & =  w_t - \alpha_t \nabla f(w_t^{md}) \lambda_t^{\star}    \\
w_{t+1}^{ag} & = \frac{A_{t-1}}{A_{t}} w_{t}^{ag} +\frac{\alpha_{t}}{A_{t}} w_{t+1}
\end{cases}
\end{aligned}
\end{equation}
With the following fixed-point property
\[ \lambda_t^{\star,\top} f(w_{t+1}^{ag}) - R(\lambda_t^{\star} ) = \max_{\lambda \in \Lambda} [\lambda \cdot  f(w_{t+1}^{ag}) - R(\lambda)] \]
\begin{proof}

At iteration $t$, we have $w_{t+1} = w_t - \alpha_t \nabla f(w_t^{md}) \lambda_t^{\star}   $, therefore
\begin{align*}
 \alpha_t \langle   \nabla f(w_t^{md}) \lambda_t^{\star}  , w_{t+1} - w_\star \rangle & =\langle w_t - w_{t+1}, w_{t+1} - w_\star \rangle   \\
 & = \frac{1}{2} \Big( \| w_\star - w_t\|_2^2 - \| w_\star - w_{t+1}\|_2^2 - \| w_{t+1} - w_t\|_2^2 \Big)  
\end{align*}
where we use the identity $\langle a,b \rangle  = \frac{1}{2} \left( \| a+b\|_2^2 - \| a\|_2^2 - \| b\|_2^2 \right)$. Then, we rearrange the term into
\begin{align*}
 \alpha_t \langle \nabla f(w_t^{md}) \lambda_t^{\star} , w_t^{md} - w_\star \rangle & = \frac{1}{2} \| w_\star - w_t \|_2^2 -\frac{1}{2} \| w_\star - w_{t+1}\|_2^2 \\
 & +  \alpha_t \langle \nabla f(w_t^{md}) \lambda_t^{\star}, w_t^{md} - w_{t+1}   \rangle  - \frac{1}{2} \| w_{t+1}-  w_t \|_2^2
\end{align*}

We apply the convexity of $\lambda_{t}^{\star} \cdot f$ on $w$ from \Cref{Assumption-linear-convex}
\begin{equation}
\begin{aligned}
\lambda_t^{\star,\top} f(w_t^{md}) -  \lambda_t^{\star,\top} f(w_\star) & \leq  \langle  \nabla f(w_t^{md}) \lambda_t^{\star}, w_t^{md} - w_\star \rangle  \\
\alpha_t \Big(  \lambda_t^{\star,\top} f(w_t^{md}) - \lambda_t^{\star,\top} f(w_\star)  \Big)   & \leq  \frac{1}{2}\| w_{\star} - w_t \|_2^2  - \frac{1}{2}\|w_{\star} - w_{t+1}\|_2^2   - \frac{1}{2} \| w_{t+1}- w_t\|_2^2   \\
& + \ \alpha_t \langle \nabla f(w_t^{md}) \lambda_t^{\star}, w_t^{md} - w_{t+1}   \rangle 
\end{aligned}
\end{equation}
After arranging, we have
\begin{equation}
\label{eq-noiseless-step1}
\begin{aligned}
\alpha_t \Big(  \lambda_t^{\star,\top} f(w_t^{md}) +  \langle \nabla f(w_t^{md}) \lambda_t^{\star} , w_{t+1} - w_{t}^{md}  \rangle \Big)   & \leq  \frac{1}{2}\| w_{\star} - w_t \|_2^2  - \frac{1}{2}\|w_{\star} - w_{t+1}\|_2^2      \\
& + \alpha_t \lambda_t^{\star,\top} f(w_{\star}) - \frac{1}{2} \| w_{t+1}- w_t\|_2^2 
\end{aligned}
\end{equation}
On the left hand side, it is related to $\lambda_{t}^{\star} \cdot f(w_{t+1}^{ag})$. We use the $L$-smoothness of $\lambda_{t}^{\star} \cdot f $ , 
\begin{align*}
\lambda_{t}^{\star} \cdot f(w_{t+1}^{ag}) & \leq \lambda_{t}^{\star} \cdot f(w_{t}^{md}) + \langle \nabla f(w_{t}^{md}) \lambda_{t}^{\star} , w_{t+1}^{ag}  - w_t^{md} \rangle   + \frac{L}{2} \| w_{t+1}^{ag} - w_t^{md}\|^2 
\end{align*}
From the fixed-point property $\lambda_t^{\star,\top} f(w_{t+1}^{ag}) + R(\lambda_t^{\star} ) = \max_{\lambda \in \Lambda} [\lambda \cdot  f(w_{t+1}^{ag}) + R(\lambda)]$, we know
\begin{equation}
\label{eq-noiseless-step-2}
\begin{aligned}
\max_{\lambda \in \Lambda }\Phi(\lambda ,w_{t+1}^{ag}) & = \lambda_t^{\star,\top} f(w_{t+1}^{ag}) - R(\lambda_t^{\star} )   \\
& \leq \lambda_{t}^{\star} \cdot f(w_{t}^{md}) + \langle \nabla f(w_{t}^{md}) \lambda_{t}^{\star} , w_{t+1}^{ag}  - w_t^{md} \rangle    \\
& + \frac{L}{2} \| w_{t+1}^{ag} - w_t^{md}\|_2^2 - R(\lambda_t^{\star} )  
\end{aligned}
\end{equation}
Now we make the following decomposition for the scalar product :
\begin{align*}
 \langle  \nabla f(w_{t}^{md}) \lambda_{t}^{\star}  , w_{t+1}^{ag}  - w_t^{md} \rangle  
 = \frac{\alpha_t}{A_t} \langle  \nabla f(w_{t}^{md}) \lambda_{t}^{\star}  , w_{t+1}  - w_t^{md} \rangle  
 +  \frac{A_{t-1}}{A_t}\langle \nabla f(w_{t}^{md}) \lambda_{t}^{\star}  , w_{t}^{ag}  - w_t^{md} \rangle 
\end{align*}
Therefore, \Cref{eq-noiseless-step-2} scaled by $A_t$ becomes
\begin{equation}
\label{eq-noiseless-step-3}
\begin{aligned}
A_t \max_{\lambda \in \Lambda }\Phi(\lambda ,w_{t+1}^{ag}) & \leq \alpha_t  \left( \lambda_t^{\star,\top} f(w_{t}^{md}) +  \langle \nabla f(w_{t}^{md}) \lambda_{t}^{\star}, w_{t+1}  - w_t^{md} \rangle  \right) \\
& + A_{t-1} \left( \lambda_t^{\star,\top} f(w_{t}^{md}) + \langle  \nabla f(w_{t}^{md}) \lambda_{t}^{\star}, w_{t}^{ag}  - w_t^{md} \rangle  \right) \\
& + \frac{LA_t}{2} \| w_{t+1}^{ag} - w_t^{md}\|_2^2 - A_t R(\lambda_t^{\star} )  
\end{aligned}
\end{equation}
We recall \Cref{eq-noiseless-step1} which contains the first line in \Cref{eq-noiseless-step-3}
\begin{align*}
\alpha_t \Big(  \lambda_t^{\star,\top} f(w_t^{md})  +  \langle\nabla f(w_{t}^{md}) \lambda_{t}^{\star} , w_{t+1} - w_{t}^{md}  \rangle \Big)   & \leq  \frac{1}{2}\| w_{\star}-w_t\|_2^2  - \frac{1}{2}\| w_{\star} - w_{t+1} \|_2^2      \\
& + \alpha_t \lambda_t^{\star}  \cdot f(w_\star) - \frac{1}{2} \| w_{t+1}- w_t\|_2^2 
\end{align*}
The second term in \Cref{eq-noiseless-step-3} is related to $\lambda_t^{\star,\top} f(w_t^{ag})$. By the convexity of $\lambda_{t+1} \cdot f$ 
\begin{align*}
A_{t-1} \left( \lambda_t^{\star,\top} f(w_{t}^{md}) + \langle  \nabla  f(w_t^{md})\lambda_t, w_{t}^{ag}  - w_t^{md} \rangle  \right) \leq A_{t-1} \lambda_t^{\star,\top} f( w_t^{ag})
\end{align*}
Using both simplifications, \Cref{eq-noiseless-step-3} becomes 
\begin{equation}
\label{eq-noiseless-step-4}
\begin{aligned}
A_t \max_{\lambda \in \Lambda }\Phi(\lambda ,w_{t+1}^{ag}) & \leq \frac{1}{2} \| w_{\star}-w_t\|_2^2  - \frac{1}{2}\| w_{\star} - w_{t+1} \|_2^2  +  \alpha_t \lambda_t^{\star,\top} f(w_\star) - \frac{1}{2} \| w_{t+1}- w_t\|^2 \\
& + A_{t-1} \lambda_t^{\star,\top} f(w_t^{ag}) + \frac{L A_t }{2} \| w_{t+1}^{ag} - w_t^{md}\|^2 - A_t R(\lambda_t^{\star} )
\end{aligned}
\end{equation}
By noticing :
\begin{align*}
w_{t+1}^{ag} - w_t^{md} = \frac{\alpha_t}{A_t} \Big( w_{t+1} - w_t \Big)
\end{align*}
We have after rearranging with $f$ and $\Phi$ and decomposing $A_t R(\lambda_t^{\star}) = (A_{t-1} + \alpha_t)R(\lambda_t^{\star})$:
\begin{equation}
\begin{aligned}
A_t  \max_{\lambda \in \Lambda }\Phi(\lambda ,w_{t+1}^{ag})  & \leq \frac{1}{2}\| w_{\star} - w_t \|_2^2 -  \frac{1}{2} \| w_{\star} - w_{t+1} \|_2^2  + \left( \frac{L\alpha_t^2}{2A_t} - \frac{1}{2} \right) \| w_{t+1}- w_t\|^2 \\
& + A_{t-1}  \max_{\lambda \in \Lambda }\Phi(\lambda, w_t^{ag}) + \alpha_t  \Phi(\lambda_t^{\star},w_\star) 
\end{aligned}
\end{equation}
By telescopic sum, we have
\begin{equation}
\label{eq-noiseless-step-5}
\begin{aligned}
A_T  \max_{\lambda \in \Lambda }\Phi(\lambda ,w_{T+1}^{ag})  & \leq \frac{1}{2} \| w_{\star} - w_{\mathrm{init}} \|_2^2 + \sum_{t=1}^{T} \left( \frac{L\alpha_t^2}{A_t} - \frac{1}{2} \right) \| w_{t+1}- w_t\|^2 \\
& + A_{0}  \max_{\lambda \in \Lambda }\Phi(\lambda, w_{\mathrm{init}}) + \sum_{t=1}^{T} \alpha_t  \Phi(\lambda_t^{\star},w_\star) \\
 & \leq \frac{1}{2} \| w_{\star} - w_{\mathrm{init}} \|_2^2 + A_{0}  \max_{\lambda \in \Lambda }\Phi(\lambda, w_{\mathrm{init}})  \\
& +  \sum_{t=1}^{T}  \left( \frac{L\alpha_t^2}{2A_t} - \frac{1}{2} \right) \| w_{t+1}- w_t\|^2  + A_t \max_{\lambda \in \Lambda}  \Phi(\lambda,w_\star) 
\end{aligned}
\end{equation}
From \Cref{lem-step-size}, the quadratic term cancels.
\end{proof}

\subsection{With noises}
\label{appendix-proof-accelerated-with-noise}

We extend the previous result with a precision $\epsilon_t$.
\begin{equation}
\begin{aligned}
\begin{cases}
w_{t}^{md} & = \frac{A_{t-1}}{A_{t}} w_{t}^{ag} + \frac{\alpha_{t}}{A_{t}} w_{t} \\
& \lambda_t^{\star}  \quad  \textrm{Fixed point from \Cref{thm-fixed-point}} \\
w_{t+1} & =  w_t - \alpha_t \nabla f(w_t^{md}) \lambda_t^{\star}    \\
w_{t+1}^{ag} & = \frac{A_{t-1}}{A_{t}} w_{t}^{ag} +\frac{\alpha_{t}}{A_{t}} w_{t+1}
\end{cases}
\end{aligned}
\end{equation}
with the following property \Cref{eq-epsilont-t}
\[ \max_{\lambda \in \Lambda}  \Phi (\lambda,w_{t+1}^{ag}) - \Phi(\lambda_{t+1}, w_{t+1}^{ag}) \leq \epsilon_t + \frac{L}{2} \|w_{t+1}^{ag} - w_t^{md} \|_2^2.\]
\begin{proof}
At iteration $t$, we have $w_{t+1} = w_t - \alpha_t \nabla f( w_t^{md})\lambda_{t+1} $, therefore
\begin{align*}
 \alpha_t \langle  \nabla f( w_t^{md})\lambda_{t+1} , w_{t+1} - w_\star \rangle & =\langle w_t - w_{t+1}, w_{t+1} - w_\star \rangle   \\
 & = \frac{1}{2} \Big( \| w_\star - w_t\|_2^2 - \| w_\star - w_{t+1}\|_2^2 - \| w_{t+1} - w_t\|_2^2 \Big)  
\end{align*}
where we use the identity $\langle a,b \rangle  = \frac{1}{2} \left( \| a+b\|_2^2 - \| a\|_2^2 - \| b\|_2^2 \right)$. Then, we rearrange the term into
\begin{align*}
 \alpha_t \langle  \nabla f( w_t^{md})\lambda_{t+1} , w_t^{md} - w_\star \rangle & = \frac{1}{2} \| w_\star - w_t \|_2^2 -\frac{1}{2} \| w_\star - w_{t+1}\|_2^2 \\
 & +  \alpha_t \langle \nabla f( w_t^{md})\lambda_{t+1} , w_t^{md} - w_{t+1}   \rangle  - \frac{1}{2} \| w_{t+1}-  w_t \|_2^2
\end{align*}

We apply the convexity of $\lambda_{t+1} f$,
\begin{equation}
\label{eq-acc-step1}
\begin{aligned}
\lambda_{t+1} f(w_t^{md}) -  \lambda_{t+1} f(w_\star) & \leq  \langle  \nabla f( w_t^{md})\lambda_{t+1} , w_t^{md} - w_\star \rangle  \\
\alpha_t \Big(  \lambda_{t+1} f(w_t^{md})  +  \langle\nabla f( w_t^{md})\lambda_{t+1}  , w_{t+1} - w_{t}^{md}  \rangle \Big)   & \leq  \frac{1}{2}\| w_{\star} - w_t \|_2^2  - \frac{1}{2}\|w_{\star} - w_{t+1}\|_2^2      \\
& + \alpha_t \lambda_{t+1} f(w_\star) - \frac{1}{2} \| w_{t+1}- w_t\|_2^2 
\end{aligned}
\end{equation}
On the left hand side, it is related to $\lambda_{t+1}f( w_{t+1}^{ag})$. We use the $L$-smoothness of $\lambda_{t+1} \cdot f$, 
\begin{align*}
\Phi(\lambda_{t+1},w_{t+1}^{ag}) & =  \lambda_{t+1} f(w_{t+1}^{ag}) - R(\lambda_{t+1})\\
& \leq \lambda_{t+1}f(w_{t}^{md}) + \langle  \nabla f( w_t^{md})\lambda_{t+1}, w_{t+1}^{ag}  - w_t^{md} \rangle  \\
& + \frac{L}{2} \| w_{t+1}^{ag} - w_t^{md}\|^2 - R(\lambda_{t+1})
\end{align*}
From the Assumption in \Cref{eq-epsilont-t},
\begin{equation}
\label{eq-acc-step-2}
\begin{aligned}
\max_{\lambda \in \Lambda }\Phi(\lambda ,w_{t+1}^{ag}) & \leq \lambda_{t+1}f(w_{t}^{md}) + \langle  \nabla f(w_t^{md})\lambda_{t+1}, w_{t+1}^{ag}  - w_t^{md} \rangle  \\
& +  L \| w_{t+1}^{ag} - w_t^{md}\|^2 + \epsilon_t - R(\lambda_{t+1})
\end{aligned}
\end{equation}
Now we make the following decomposition for the scalar product :
\begin{equation}
\begin{aligned}
 \langle  \nabla f(w_t^{md})\lambda_{t+1}, w_{t+1}^{ag}  - w_t^{md} \rangle  & 
 = \frac{\alpha_t}{A_t} \langle   \nabla f(w_t^{md})\lambda_{t+1}, w_{t+1}  - w_t^{md} \rangle  \\
 & +  \frac{A_{t-1}}{A_t}\langle   \nabla f(w_t^{md})\lambda_{t+1}, w_{t}^{ag}  - w_t^{md} \rangle 
\end{aligned}
\end{equation}
Therefore, \Cref{eq-acc-step-2} scaled by $A_t$ becomes
\begin{equation}
\label{eq-acc-step-3}
\begin{aligned}
A_t \max_{\lambda \in \Delta }\Phi(\lambda ,w_{t+1}^{ag}) & \leq \alpha_t  \left( \lambda_{t+1} f(w_{t}^{md}) +  \langle  \nabla f(w_t^{md})\lambda_{t+1}, w_{t+1}  - w_t^{md} \rangle  \right) \\
& + A_{t-1} \left(\lambda_{t+1}f(w_{t}^{md}) + \langle  \nabla f(w_t^{md})\lambda_{t+1}, w_{t}^{ag}  - w_t^{md} \rangle  \right) \\
& + L A_t \| w_{t+1}^{ag} - w_t^{md}\|^2 + A_t \epsilon_t - A_t R(\lambda_{t+1})
\end{aligned}
\end{equation}
The first line in \Cref{eq-acc-step-3}  appears in \Cref{eq-acc-step1} from Gradient descent.
\begin{align*}
\alpha_t \Big(  \lambda_{t+1}f(w_t^{md})  +  \langle \nabla f(w_t^{md})\lambda_{t+1} , w_{t+1} - w_{t}^{md}  \rangle \Big)   & \leq  \frac{1}{2}\| w_{\star}-w_t\|_2^2  - \frac{1}{2}\| w_{\star} - w_{t+1} \|_2^2      \\
& + \alpha_t \lambda_{t+1} f(w_\star) - \frac{1}{2} \| w_{t+1}- w_t\|^2 
\end{align*}
The second term in \Cref{eq-acc-step-3} is related to $\lambda_{t+1}f(w_t^{ag})$. By the convexity of $\lambda_{t+1}f$
\begin{align*}
A_{t-1} \left( \lambda_{t+1}f(w_{t}^{md}) + \langle  \nabla f( w_t^{md})\lambda_{t+1}, w_{t}^{ag}  - w_t^{md} \rangle  \right) \leq A_{t-1} \lambda_{t+1} f(w_t^{ag})
\end{align*}
Using both simplifications, \Cref{eq-acc-step-3} becomes 
\begin{equation}
\label{eq-acc-step-4}
\begin{aligned}
A_t \max_{\lambda \in \Delta }\Phi(\lambda ,w_{t+1}^{ag}) & \leq \frac{1}{2} \| w_{\star}-w_t\|_2^2  - \frac{1}{2}\| w_{\star} - w_{t+1} \|_2^2  +  \alpha_t \lambda_{t+1} f(w_\star) - \frac{1}{2} \| w_{t+1}- w_t\|^2 \\
& + A_{t-1} \lambda_{t+1} f(w_t^{ag}) + L A_t  \| w_{t+1}^{ag} - w_t^{md}\|^2 + A_t \epsilon_t - A_t R(\lambda_{t+1})
\end{aligned}
\end{equation}
By noticing :
\begin{align*}
w_{t+1}^{ag} - w_t^{md} & = \frac{\alpha_t}{A_t} \Big( w_{t+1} - w_t \Big) \\
A_t R(\lambda_{t+1}) & = (A_{t-1} + \alpha_t)R(\lambda_{t+1})
\end{align*}
We have after rearranging :
\begin{equation}
\label{eq-main-acceleration-results}
\begin{aligned}
A_t  \max_{\lambda \in \Delta }\Phi(\lambda ,w_{t+1}^{ag})  & \leq \frac{1}{2}\| w_{\star} - w_t \|_2^2 -  \frac{1}{2} \| w_{\star} - w_{t+1} \|_2^2  + \left( \frac{L\alpha_t^2}{A_t} - \frac{1}{2} \right) \| w_{t+1}- w_t\|^2 \\
& + A_{t-1} \Big( \lambda_{t+1}f(w_t^{ag}) - R(\lambda_{t+1}) \Big) + \alpha_t  \Big( \lambda_{t+1} f(w_\star)  -R(\lambda_{t+1}) \Big) +  A_t \epsilon_t \\
& \leq \frac{1}{2}\| w_{\star} - w_t \|_2^2 -  \frac{1}{2} \| w_{\star} - w_{t+1} \|_2^2  + \left( \frac{L\alpha_t^2}{A_t} - \frac{1}{2} \right) \| w_{t+1}- w_t\|^2 \\
& + A_{t-1}  \max_{\lambda \in \Lambda }\Phi(\lambda, w_t^{ag}) + \alpha_t  \Phi(\lambda_{t+1},w_\star)  +  A_t \epsilon_t 
\end{aligned}
\end{equation}
From \Cref{lem-step-size}, the quadratic terms cancel and we conclude by applying telescopic sum as before.
\end{proof}

\section{Cross-evaluation gap}
\label{appendix-proof-cross-gap}

We are demonstrating \Cref{thm-gap} which shows that $\forall \lambda,\mu \in \Lambda$
\begin{align*}
\Phi(\mu, w_{t+1}^{ag}(\lambda)) - \Phi(\lambda, w_{t+1}^{ag}(\lambda)) \le \langle \nabla S_t(\lambda),\lambda-\mu \rangle + R(\lambda) - R(\mu) + \frac{L }{2 } \| w_{t+1}^{ag}(\lambda) - w_t^{md} \|_2^2
\end{align*}
\begin{proof}
We recall that $V_t(\lambda) := \nabla_w \Phi(\lambda,w_t^{md})$. By the definition of the Nesterov aggregate sequence, we have:
\begin{equation}
\label{eq-gap-step1}
\begin{aligned}
w_{t+1}^{ag}(\lambda) & = w_t^{md} + \frac{\alpha_t}{A_t} (w_{t+1}(\lambda) - w_t) \\
w_{t+1}^{ag}(\lambda) - w_t^{md} & = -\frac{\alpha_t^2}{A_t} V_t(\lambda) = -\frac{\alpha_t^2}{A_t}  \nabla f(w_t^{md}) \lambda 
\end{aligned}
\end{equation}
Using the $L$-smoothness of $\Phi$ with respect to the primal variable $w$ ,we upper-bound the function value for any chosen adversary $\mu \in \Lambda$:
\[\Phi(\mu, w_{t+1}^{ag}(\lambda)) \le \Phi(\mu, w_t^{md}) + \langle V_t(\mu), w_{t+1}^{ag}(\lambda) - w_t^{md} \rangle + \frac{L}{2} \vert{}\vert{}w_{t+1}^{ag}(\lambda) - w_t^{md}\vert{}\vert{}_2^2\]

Simultaneously, using the convexity of $\Phi$ with respect to $w$ for the dual variable $\lambda \in \Lambda$, we lower-bound its value:
\[\Phi(\lambda, w_{t+1}^{ag}(\lambda)) \ge \Phi(\lambda, w_t^{md}) + \langle V_t(\lambda), w_{t+1}^{ag}(\lambda) - w_t^{md} \rangle\]

Subtracting the convexity inequality from the smoothness inequality isolates the cross-evaluation gap:
\begin{equation}
\label{eq-gap-step2}
\begin{aligned}
\Phi(\mu, w_{t+1}^{ag}(\lambda)) - \Phi(\lambda, w_{t+1}^{ag}(\lambda)) & \le \Phi(\mu, w_t^{md}) - \Phi(\lambda, w_t^{md}) \\
& + \langle V_t(\mu) - V_t(\lambda), w_{t+1}^{ag}(\lambda) - w_t^{md} \rangle + \frac{L}{2} \vert{}\vert{}w_{t+1}^{ag}(\lambda) - w_t^{md}\vert{}\vert{}_2^2
\end{aligned}
\end{equation}
We now substitute \Cref{eq-gap-step1} into the inner product term:
\begin{align*}
\langle V_t(\mu) - V_t(\lambda), w_{t+1}^{ag}(\lambda) - w_t^{md} \rangle & = -\frac{\alpha_t^2}{A_t} \langle \nabla f(w_t^{md})(\mu - \lambda) ,  \nabla f(w_t^{md}) \lambda  \rangle  \\
& =  \frac{\alpha_t^2}{A_t} \langle  \lambda - \mu,   \nabla f(w_t^{md})^{\top} \nabla f(w_t^{md}) \lambda \rangle 
\end{align*}
Now we decompose $\Phi$,
\begin{align*}
\Phi(\mu, w_t^{md}) - \Phi(\lambda, w_t^{md}) = (\mu-\lambda)\cdot f(w_t^{md}) + R(\lambda ) - R(\mu) 
\end{align*}
Substituting these two lines into \Cref{eq-gap-step2}, we obtain
\begin{align*}
\Phi(\mu, w_{t+1}^{ag}(\lambda)) - \Phi(\lambda, w_{t+1}^{ag}(\lambda)) & \le \langle \frac{\alpha_t^2}{A_t} \nabla f(w_t^{md})^{\top} \nabla f(w_t^{md}) \lambda - f(w_t^{md})  , \lambda -\mu \rangle \\
&  + \frac{L}{2} \vert{}\vert{}w_{t+1}^{ag}(\lambda) - w_t^{md}\vert{}\vert{}_2^2 + R(\lambda) - R(\mu)
\end{align*}
we conclude by noticing from \Cref{eq-moreau-env} that 
\[\nabla S_t(\lambda) = \frac{\alpha_t^2}{A_t} \nabla f(w_t^{md})^{\top} \nabla f(w_t^{md}) \lambda  - f(w_t^{md}).\]
\end{proof}
\subsection{General upper bound for smooth concave-convex function}
In this subsection, we suppose that $\Phi$ is a general smooth concave-convex function, the general form is the Moreau envelop 
\begin{align*}
V_t(\lambda) & := \nabla_w \Phi(\lambda, w_t^{md})  \\
h_t(\lambda) & :=  \frac{\alpha_t^2}{2 A_t} \Vert{}V_t(\lambda)\Vert{}_2^2 - \Phi(\lambda, w_t^{md}) 
\end{align*}
We will show the new upper bound in the general setting
\begin{theorem}
\label{thm-ht-general}
For any $\lambda, \mu \in \Lambda$, the cross-evaluation gap is strictly upper-bounded by:
\begin{align*}
\Phi(\mu, g_t(\lambda)) - \Phi(\lambda, g_t(\lambda)) \le \Big[ h_t(\lambda) - h_t(\mu) \Big] + \frac{\alpha_t^2}{2 A_t} \Big\Vert{} V_t(\mu) - V_t(\lambda) \Big\Vert{}_2^2 + \frac{L \alpha_t^4}{2 A_t^2} \Big\Vert{} V_t(\lambda) \Big\Vert{}_2^2
\end{align*}
where $h_t$ refers to \Cref{eq-moreau-env} and $g_t(\lambda):=w_{t+1}^{ag}(\lambda)$.
\end{theorem}
\begin{proof}
By the definition of the Nesterov aggregate sequence, we have:
\begin{equation}
\label{eq-ht-general-1}
\begin{aligned}
w_{t+1}^{ag}(\lambda) & = w_t^{md} + \frac{\alpha_t}{A_t} (w_{t+1}(\lambda) - w_t) \\
w_{t+1}^{ag}(\lambda) - w_t^{md} & = -\frac{\alpha_t^2}{A_t} V_t(\lambda) 
\end{aligned}
\end{equation}
Using the $L$-smoothness of $\Phi$ with respect to the primal variable $w$, we upper-bound the function value for any chosen adversary $\mu \in \Lambda$:
\[\Phi(\mu, w_{t+1}^{ag}(\lambda)) \le \Phi(\mu, w_t^{md}) + \langle V_t(\mu), w_{t+1}^{ag}(\lambda) - w_t^{md} \rangle + \frac{L}{2} \vert{}\vert{}w_{t+1}^{ag}(\lambda) - w_t^{md}\vert{}\vert{}_2^2\]

Simultaneously, using the convexity of $\Phi$ with respect to $w$ for the dual variable $\lambda \in \Lambda$, we lower-bound its value:
\[\Phi(\lambda, w_{t+1}^{ag}(\lambda)) \ge \Phi(\lambda, w_t^{md}) + \langle V_t(\lambda), w_{t+1}^{ag}(\lambda) - w_t^{md} \rangle\]

Subtracting the convexity inequality from the smoothness inequality isolates the cross-evaluation gap:
\begin{align*}
\Phi(\mu, w_{t+1}^{ag}(\lambda)) - \Phi(\lambda, w_{t+1}^{ag}(\lambda)) & \le \Phi(\mu, w_t^{md}) - \Phi(\lambda, w_t^{md}) \\
& + \langle V_t(\mu) - V_t(\lambda), w_{t+1}^{ag}(\lambda) - w_t^{md} \rangle + \frac{L}{2} \vert{}\vert{}w_{t+1}^{ag}(\lambda) - w_t^{md}\vert{}\vert{}_2^2
\end{align*}
We now substitute \Cref{eq-ht-general-1} into the inner product term:
\begin{align*}
\langle V_t(\mu) - V_t(\lambda), w_{t+1}^{ag}(\lambda) - w_t^{md} \rangle & = -\frac{\alpha_t^2}{A_t} \langle V_t(\mu) - V_t(\lambda), V_t(\lambda) \rangle \\
& = -\frac{\alpha_t^2}{A_t} \langle V_t(\mu), V_t(\lambda) \rangle + \frac{\alpha_t^2}{A_t} \vert{}\vert{}V_t(\lambda)\vert{}\vert{}_2^2
\end{align*}
By employing the standard algebraic identity $-2\langle a, b \rangle = \vert{}\vert{}a - b\vert{}\vert{}_2^2 - \vert{}\vert{}a\vert{}\vert{}_2^2 - \vert{}\vert{}b\vert{}\vert{}_2^2$, we rewrite the inner product dynamically:
\begin{align*}
-\frac{\alpha_t^2}{A_t} \langle V_t(\mu), V_t(\lambda) \rangle = \frac{\alpha_t^2}{2A_t} \vert{}\vert{}V_t(\mu) - V_t(\lambda)\vert{}\vert{}_2^2 - \frac{\alpha_t^2}{2A_t} \vert{}\vert{}V_t(\mu)\vert{}\vert{}_2^2 - \frac{\alpha_t^2}{2A_t} \vert{}\vert{}V_t(\lambda)\vert{}\vert{}_2^2
\end{align*}
Re-adding the $\frac{\alpha_t^2}{A_t} \vert{}\vert{}V_t(\lambda)\vert{}\vert{}_2^2$ expansion component, the total cross-term equates strictly to:
\begin{align*}
\langle V_t(\mu) - V_t(\lambda), w_{t+1}^{ag}(\lambda) - w_t^{md} \rangle = \frac{\alpha_t^2}{2A_t} \vert{}\vert{}V_t(\mu) - V_t(\lambda)\vert{}\vert{}_2^2 - \frac{\alpha_t^2}{2A_t} \vert{}\vert{}V_t(\mu)\vert{}\vert{}_2^2 + \frac{\alpha_t^2}{2A_t} \vert{}\vert{}V_t(\lambda)\vert{}\vert{}_2^2
\end{align*}

Substitute this resulting identity back into the main cross-evaluation gap inequality and group the variables by $\lambda$ and $\mu$:
\begin{align*}
\Phi(\mu, w_{t+1}^{ag}(\lambda)) - \Phi(\lambda, w_{t+1}^{ag}(\lambda)) & \le \left( \frac{\alpha_t^2}{2A_t} \vert{}\vert{}V_t(\lambda)\vert{}\vert{}_2^2 - \Phi(\lambda, w_t^{md}) \right) - \left( \frac{\alpha_t^2}{2A_t} \vert{}\vert{}V_t(\mu)\vert{}\vert{}_2^2 - \Phi(\mu, w_t^{md}) \right) \\ & + \frac{\alpha_t^2}{2A_t} \vert{}\vert{}V_t(\mu) - V_t(\lambda)\vert{}\vert{}_2^2 + \frac{L}{2} \vert{}\vert{}w_{t+1}^{ag}(\lambda) - w_t^{md}\vert{}\vert{}_2^2
\end{align*}

We recognize by \Cref{eq-moreau-env} that the first two terms correspond to  Moreau Envelope Proxy. Then, we address the terminal $L$-smoothness bounding term  $w_{t+1}^{ag}(\lambda) - w_t^{md} = \frac{\alpha_t}{A_t} (w_{t+1}(\lambda) - w_t)$. Therefore:
\[\frac{L}{2} \vert{}\vert{}w_{t+1}^{ag}(\lambda) - w_t^{md}\vert{}\vert{}_2^2 = \frac{L \alpha_t^2}{2A_t^2} \vert{}\vert{}w_{t+1}(\lambda) - w_t\vert{}\vert{}_2^2 = \frac{L\alpha_t^4}{2A_t^2} \| V_t (\lambda)\|_2^2 \]
Combining all simplified blocks yields the desired bounding theorem exactly.
\end{proof}

\section{KKT condition for quadratic case}
\label{sec-kkt-pure-quadratic}
We recall \Cref{lem-kkt-pure-quadratic}.
\begin{lemma}
Suppose an interior-point algorithm terminates at a strictly feasible output $\mu_k \in\text{relint}(\Delta_J)$, returning a barrier parameter $1 /\tau_k \le \epsilon_\tau$ and dual multipliers $\nu_k \in \mathbb{R}$ and $z_k \in \mathbb{R}_{+}^{J}$ that satisfy the following inexact Karush-Kuhn-Tucker (KKT) conditions:
\begin{itemize}
    \item \textbf{Approximate Stationarity} $\vert{}\vert{}\nabla h_{t}(\mu_{k})-\nu_{k}1-z_{k}\vert{}\vert{}_{\infty}\le\epsilon_{g} $
    \item \textbf{Perturbed Complementary Slackness} $\mu_{k,j} z_{k,j} = 1/\tau_{k}$ for all $j\in\{1,...,J\}$
\end{itemize}
Then, the Frank-Wolfe duality gap at $\mu_k$ is strictly bounded by:
\begin{align*}
\max_{\mu \in \Lambda} \langle \nabla h_t(\mu_k), \mu_k - \mu \rangle \le J\epsilon_{\tau}+2\epsilon_{g}
\end{align*}
\end{lemma}

\begin{proof}
Because $h_{t}$ is convex, the true suboptimality is upper-bounded by the Frank-Wolfe gap at $\mu_{k}$:
\begin{align*}
\max_{\mu \in \Lambda} \langle \nabla h_t(\mu_k), \mu_k - \mu \rangle = \langle\nabla h_{t}(\mu_{k}),\mu_{k}\rangle-\min_{\mu \in\Delta_{J}}\langle\nabla h_{t}(\mu_{k}),\mu\rangle
\end{align*}
Let $r_{k}=\nabla h_{t}(\mu_{k})-\nu_{k}1-z_{k}$ denote the stationarity residual, by assumption we have $\vert{}\vert{}r_{k}\vert{}\vert{}_{\infty}\le\epsilon_{g}$.
Substituting $\nabla h_{t}(\mu_{k})=\nu_{k}1+z_{k}+r_{k}$ yields
\begin{align*}
\langle\nabla h_{t}(\mu_{k}),\mu_{k}\rangle & =\nu_{k}\langle \textbf{1},\mu_{k}\rangle+\langle z_{k},\mu_{k}\rangle+\langle r_{k},\mu_{k}\rangle \\
& = \nu_{k}+J/\tau_{k}+\langle r_{k},\mu_{k}\rangle
\end{align*}
Then, we notice that
\begin{align*}
\min_{y\in\Delta_{J}}\langle\nabla h_{t}(\mu_{k}),\mu \rangle& = \min_{y\in\Delta_{J}}\langle\nu_{k}1+z_{k}+r_{k},\mu \rangle = \nu_{k}+\min_{j}(z_{k,j}+r_{k,j}) \\
\max_{\mu \in \Lambda} \langle \nabla h_t(\mu_k), \mu_k - \mu \rangle  & =J/\tau_{k}+\langle r_{k},\mu_{k}\rangle-\min_{j}(z_{k,j}+r_{k,j})
\end{align*}
Since $\vert{}\vert{}r_{k}\vert{}\vert{}_{\infty}\le\epsilon_{g}$, by Holder inequality:
\begin{align*}
\langle r_{k},\mu_{k}\rangle\le\vert{}\vert{}r_{k}\vert{}\vert{}_{\infty}\vert{}\vert{}\mu_{k}\vert{}\vert{}_{1}& \le\epsilon_{g} \\
\forall 1\leq j\leq J, \quad z_{k,j}+r_{k,j}> r_{k,j} \ge-\vert{}\vert{}r_{k}\vert{}\vert{}_{\infty} & \ge-\epsilon_{g} \\
\implies -\min_{j}(z_{k,j}+r_{k,j})& \le\epsilon_{g}
\end{align*}
Since the algorithm ensures $1/\tau_{k}\le\epsilon_{\tau}$, we conclude:
\[ g_{FW}(\mu_{k})\le J\epsilon_{\tau}+2\epsilon_{g}.\]
\end{proof}

\section{Path-following algorithm}
\label{appendix-simplex}

We rewrite the path-following algorithm \Cref{alg:two_phase_pfm} from \citep{nesterov1994interior}.
\begin{algorithm}
\caption{Centered path-following method}
\label{alg:two_phase_pfm}
\begin{algorithmic}[1]
\State \textbf{Initialize:} Set $\mu_0 \in \arg\min R+B$ and $\tau_0 = 1$.
\State \textbf{Phase I:} \text{Initial Centering (Damped Step)}
\While{ $\delta(\mu_k, F_1) > 1/4$}
    \State Compute Newton Step $\Delta \mu_k$ for $(\mu_k, F_1)$ with \Cref{eq-solution-kkt}
    \State Update: $\mu_{k+1} = \mu_k + \frac{1}{1+\delta(\mu_k, F_1)}\Delta \mu_k$ , $k \xleftarrow{} k+1 $
\EndWhile
\State \textbf{Phase II:} \text{ Short-Step Path-Following(Normal Full Step)}
\While{Target accuracy $\epsilon_t$ is not reached}
    \State Update Penalty: $\tau_{k+1} = \tau_k \times \left(1 + \frac{1}{15\sqrt{J}}\right)$
    \State Compute Newton Step $\Delta \mu_k$ for $(\mu_k, F_{\tau_{k+1}})$ with \Cref{eq-solution-kkt}
    \State Update: $\mu_{k+1} = \mu_k + \Delta \mu_k$ (, $k \xleftarrow{} k+1 $
\EndWhile
\State \Return $\mu_k$
\end{algorithmic}
\end{algorithm}
From Nesterov framework, we show in the following Theorem about equality-constrained Newton method.
\begin{theorem}
\label{thm:constrained_newton}
Consider $F:\mathbb{R}_{++}^J \xrightarrow{} \mathbb{R}$ strictly convex and $1$-self-concordant, let $\mu_k $ be a strictly feasible solution and $(\Delta \mu_k,s_k)$ be defined as the solution of the KKT block system:
\begin{equation}
\begin{bmatrix} \nabla^2 F(\mu_k) & \mathbf{1} \\ \mathbf{1}^\top & 0 \end{bmatrix} \begin{bmatrix} \Delta\mu_k \\ s_k \end{bmatrix} = \begin{bmatrix} -\nabla F(\mu_k) \\ 0 \end{bmatrix}
\end{equation}
Let $\omega(\lambda) = \lambda - \ln(1+\lambda)$ and $\lambda_k = \delta(\mu_k, F) $, then the following properties hold for the constrained updates.
\begin{enumerate}
    \item \textbf{Damped Newton Step:} Taking the step $\mu_{k+1} = \mu_k + \frac{1}{1+\lambda_k}\Delta\mu_k$ maintains strict feasibility, and the objective strictly decreases by:
    \begin{equation}
        F(\mu_{k+1}) \le F(\mu_k) - \omega(\lambda_k)
    \end{equation}
    \item \textbf{Full Newton Step:} If $\delta(\mu_k, F) < 1$, taking a full Newton step $\mu_{k+1} = \mu_k + \Delta\mu_k$ maintains strict feasibility, converges quadratically, and the new Newton decrement is bounded by:
    \begin{equation}
        \delta(\mu_{k+1}, F) \le \left(\frac{\delta(\mu_k, F)}{1-\delta(\mu_k, F)}\right)^2
    \end{equation}
\end{enumerate}
\end{theorem}
The proof can be found found in Theorem 5.1.15, 5.2.2 in \cite{nesterov2018lectures} and Section 10.2.3 in \cite{boyd2004convex}.

If $\nabla^2 F_\tau > 0$, we can efficiently calculate the step solution $(\Delta\mu, s)$ via block elimination:
\begin{equation}
\label{eq-solution-kkt}
\begin{aligned}
s &\leftarrow -\frac{\mathbf{1}^\top (\nabla^2 F_\tau(\mu))^{-1} \nabla F_\tau(\mu)}{\mathbf{1}^\top (\nabla^2 F_\tau(\mu))^{-1} \mathbf{1}}, \\
\Delta\mu &\leftarrow -(\nabla^2 F_\tau(\mu))^{-1} (\nabla F_\tau(\mu) + s\mathbf{1}).
\end{aligned}
\end{equation}

\subsection{Cross-evaluation gap bound}
\label{proof-kkt-upper-bound}
Before the main result, we cite a property of self-concordant-barrier from \citep{nesterov1994interior}.
\begin{lemma}
\label{lem-barrier}
If $B_\Lambda: \operatorname{relint}(\Lambda) \to \mathbb{R}$ is a $\nu$-self-concordant barrier for the compact convex set $\Lambda$. Then for all $\mu \in \operatorname{relint}(\Lambda)$ and $\lambda \in \Lambda$:
\begin{align*}
\langle \nabla B_\Lambda(\mu), \lambda - \mu \rangle \le \nu \\
\Vert{}\nabla B_\Lambda(\mu)\Vert{}_{\nabla^2 B_\Lambda(\mu)}^* \le \sqrt{\nu}
\end{align*}
\end{lemma}

We recall the main Theorem to prove.
\begin{theorem}[Inner problem Upper bound via KKT block]
Under \Cref{Assumption-linear-convex}, we consider a fixed $\tau \geq 1 $ and  $\mu_k \in  \mathrm{relint}(\Lambda)$. If the KKT block on
on $(\mu_k, F_\tau)$ has a feasible solution $(\Delta \mu_k, s_k)$ and there exists $ \beta \in [0,1]$,
\begin{align*}
\delta(\mu_k,F_{\tau})\leq \beta^2,
\end{align*}
then for all $\lambda \in \Lambda $
\begin{equation}
\epsilon_t(\mu_k,\lambda) \leq 
\frac{\nu}{\tau} + \frac{\sqrt{\nu}}{\tau} \left(\frac{\beta}{1-\beta}\right) + \frac{\omega_*(\beta)}{\tau} + \frac{\beta}{1-\beta} \frac{2G\alpha_t}{\sqrt{\tau A_t}}
\end{equation}
where $\omega_*(\beta) := -\beta - \ln(1-\beta)$.
\end{theorem}

\begin{proof}
We recall that
\begin{align*}
\epsilon_t(\lambda,\mu) & = \langle \nabla S_t(\lambda), \lambda - \mu \rangle + R(\lambda) - R(\mu) \\
F_{\tau} & = \tau (S_t + R) + B
\end{align*}
We can rewrite for $\lambda, \mu^\star \in \Lambda$
\begin{equation}
\label{eq-thm-kkt-step1}
\begin{aligned}
\epsilon_t(\mu_k,\lambda) & = \langle \nabla S_t(\mu_k), \mu_k - \lambda \rangle + R(\mu_k) - R(\lambda) \\
& =\langle \nabla S_t(\mu_k) - \nabla S_t(\mu^{\star}), \mu_k - \lambda \rangle + \langle \nabla S_t(\mu^{\star}), \mu_k - \lambda \rangle + R(\mu_k) - R(\lambda).
\end{aligned}
\end{equation}
Now we fix $\mu^\star := \arg\min_{\mu \in \Lambda} F_{\tau}$, since $\mu^{\star},\mu_k \in \mathrm{relint}(\Lambda)$, for all $\lambda \in \Lambda, \mu_k - \lambda $ is a feasible position and therefore $\tau > 0, \lambda \in \Lambda$:
\begin{align*}
& \langle  \nabla F_{\tau}(\mu^\star), \mu_k -\lambda \rangle = 0 \\
\implies \quad & \langle \nabla S_t(\mu^\star), \mu_k - \lambda \rangle  = -\langle \nabla R(\mu^\star), \mu_k - \lambda \rangle - \frac{1}{\tau} \langle \nabla B(\mu^\star), \mu_k - \lambda \rangle
\end{align*}
We notice that
\begin{align*}
    R(\mu_k) - R(\lambda) -\langle \nabla R(\mu^\star), \mu_k - \lambda \rangle  = D_R(\mu_k,\mu^\star) - D_R(\lambda,\mu^\star) 
\end{align*}
Therefore, \Cref{eq-thm-kkt-step1} can be rewritten as 
\begin{equation}
\label{eq-thm-kkt-step2}
\begin{aligned}
\epsilon_t(\mu_k,\lambda) & = \langle \nabla S_t(\mu_k) - \nabla S_t(\mu^{\star}), \mu_k - \lambda \rangle - \frac{1}{\tau} \langle \nabla B(\mu^\star), \mu_k - \lambda \rangle + D_R(\mu_k,\mu^\star) - D_R(\lambda,\mu^\star) \\
& \leq \langle \nabla S_t(\mu_k) - \nabla S_t(\mu^{\star}), \mu_k - \lambda \rangle - \frac{1}{\tau} \langle \nabla B(\mu^\star), \mu_k - \lambda \rangle + D_R(\mu_k,\mu^\star)  .
\end{aligned}
\end{equation}

For the first quadratic term, because $\nabla S_t(\mu_k) - \nabla S_t(\mu^\star) = H_Q(\mu_k - \mu^\star)$ where $H_Q = \frac{\alpha_t^2}{A_t} \nabla f(w_t^{md})^{\top}\nabla f(w_t^{md})$.

We use the generalized Cauchy-Schwarz inequality. Since $\tau H_Q \preceq \nabla^2 F_\tau$, we have 
\[\Vert{}\mu_k - \mu^\star\Vert{}_{H_Q} \le \frac{1}{\sqrt{\tau}} \Vert{}\mu_k - \mu^\star\Vert{}_{\nabla^2 F_\tau(\mu^\star)} \le \frac{1}{\sqrt{\tau}} \frac{\beta}{1-\beta}.\] 
And from \Cref{Assumption-linear-convex}, $\| \nabla f(w_t^{md})\lambda\|_2 \leq G$, 
\begin{align*}
\| \mu_k - \lambda \|_{H_Q} = \frac{\alpha_t}{\sqrt{A_t}} \| \nabla f(w_t^{md}) (\mu_k - \lambda)\|_2 \leq \frac{2G\alpha_t}{\sqrt{A_t}}
\end{align*}
Therefore,
\[\langle H_Q(\mu_k - \mu^\star), \mu_k - \lambda \rangle \le \frac{\beta}{1-\beta} \frac{2G\alpha_t}{\sqrt{\tau A_t}} \]
For the second Barrier term, from \Cref{lem-barrier}, we notice that for any arbitrary $\mu^{\star} \in \mathrm{relint}(\Lambda),\lambda \in \Lambda$, 
\begin{align*}
\langle \nabla B(\mu^\star), \lambda  - \mu^\star \rangle \le \nu
\end{align*}
For the remaining difference $- \frac{1}{\tau} \langle \nabla B(\mu^\star), \mu_k - \mu^\star \rangle$, by self-concordance theory \cite{nesterov1994interior}, Theorem 5.1.8 (Equation 5.1.28),  if $\delta(\mu_k,F_{\tau}) \leq \beta < 1 $ the distance to the exact optimum $\mu^\star $ strictly is bounded by
\begin{align*}
\| \mu_k - \mu^\star\|_{\nabla^2 F_{\tau}(\mu^{\star)}} \leq \frac{\beta}{1-\beta},
\end{align*}
and the dual norm of $\| \nabla B(\mu^{\star}) \|_\ast \leq \sqrt{\nu}$ (\Cref{lem-barrier}),  therefore
\[ - \frac{1}{\tau} \langle \nabla B(\mu^{\star}), \mu_k - \lambda \rangle  \le \frac{\nu}{\tau} + \frac{\sqrt{\nu}}{\tau} \frac{\beta}{1-\beta}\]

For the last term, we notice , $D_{F_\tau}(\mu_k, \mu^\star) \le \omega_*(\beta) $ from \citep{nesterov1994interior}, therefore
\begin{align*}
D_R(\mu_k, \mu^\star) \le \frac{1}{\tau} D_{F_\tau}(\mu_k, \mu^\star) \le \frac{\omega_*(\beta)}{\tau}
\end{align*}

\end{proof}

\subsection{Phase I complexity}

\begin{theorem}[Phase I Complexity]
\label{thm-phase1}
Let $\tau_0 = 1$. Under \Cref{Assumption-linear-convex,ass:init}, starting from $\mu_0 = \frac{1}{J}\mathbf{1}$, the Damped Newton method strictly reaches the short-step neighborhood $\mathcal{N}(\beta) = \{\mu \in \operatorname{relint}(\Delta_J) \mid \delta(\mu, F_1) \le \beta\}$ with $\beta \leq 1$ in a bounded number of iterations:
\begin{align*}
N_{\text{Phase I}} \le \frac{1}{\omega(1/4)} \left(  \frac{G^2 \alpha_t^2}{2A_t} +  2 G \|w_t^{md} - w_\star\|_2 + S_\star \right)
\end{align*}
\end{theorem}

\begin{proof}

First we recall that $S_t(\mu)$ is quadratic and $F_1$ is strictly $1$-self-concordant. By \Cref{thm:constrained_newton}, we know that for the each damped Newton step $\mu_{\text{new}} = \mu + \frac{1}{1+\delta(\mu)}\Delta\mu$, we have
\begin{align*}
F_1(\mu) - F_1(\mu_{\text{new}}) \ge \omega(\lambda_k) \geq \omega(1/4)
\end{align*}
where $\omega(\beta) = \beta - \ln(1+\beta) > 0$.
Consequently, the maximum number of damped iterations required to enter the neighborhood is strictly bounded by the total initial suboptimality gap divided by this minimum step decrease:
\begin{align*}
N_{\text{Phase I}} \le \frac{F_1(\mu_0) - F_1(\mu_1^\star)}{\omega(1/4)}
\end{align*}
where $\mu_1^\star = \arg\min_{\mu \in \Delta_J} F_1(\mu)$. We notice that,
\[F_1(\mu_0) - F_1(\mu_1^\star) = \Big[ S_t(\mu_0) - S_t(\mu_1^\star) \Big] + \Big[ (R+B)(\mu_0) - (R+B)(\mu_1^*) \Big].\]
Since $\mu_0 \in \arg\min R+B$, we have
\begin{align*}
(R+B)(\mu_0) - (R+B)(\mu_1^\star) & \le 0 \\
F_1(\mu_0) - F_1(\mu_1^\star) &  \le S_t(\mu_0) - S_t(\mu_1^\star).
\end{align*}
We recall that :
\begin{align*}
S_t(\lambda) & = \frac{\alpha_t^2}{2A_t} \|  \nabla f(w_t^{md}) \lambda \|_2^2  - \lambda^{\top} f(w_t^{md}) \geq  - \lambda^{\top} f(w_t^{md})  \\
 S_t(\mu_0) - S_t(\mu_1^\star) & \leq  \frac{\alpha_t^2}{2A_t} \|  \nabla f(w_t^{md}) \mu_0 \|_2^2  + (\mu_1^\star - \mu_0)^{\top} f(w_t^{md}).
\end{align*}
We apply triangular inequality and obtain that
\begin{align*}
(\mu_1^\star - \mu_0)^{\top} f(w_t^{md}) \leq 2 \max_{j,k}|f_j(w_t^{md}) - f_k(w_t^{md})| \leq 2 G \|w_t^{md} - w_\star\|_2 + S_\star
\end{align*}
and we conclude that 
\begin{align*}
 \frac{\alpha_t^2}{2A_t} \|  \nabla f(w_t^{md}) \mu_0 \|_2^2 & \leq \frac{G^2 \alpha_t^2}{2A_t} \\
 F_1(\mu_0) - F_1(\mu_1^\star) & \leq \frac{G^2 \alpha_t^2}{2A_t} +  2 G \|w_t^{md} - w_\star\|_2 + S_\star.
\end{align*}
\end{proof}

\subsection{Phase II complexity}
We recall that $F_{\tau_k}(\mu) := \tau_k \Big(S_t(\mu)+ R(\mu)\Big) + B(\mu)$  where $B(\mu) = -\sum_{j=1}^J \ln(\mu_j)$.

\begin{theorem}[Phase II Complexity]
\label{thm-phase2}
Suppose the current iterate $\mu_k \in \operatorname{relint}(\Delta_J)$ satisfies the Newton decrement neighborhood condition 
\[\delta(\mu_k, F_{\tau_k}) \leq 1/4,\] 
If we update $\tau_{k+1} = r \tau_k$ with $r = 1 + \frac{1}{15\sqrt{J}}$, then taking exactly a full Newton step guarantees the new iterate strictly remains in the neighborhood: $\delta(\mu_{k+1},F_{\tau_{k+1}}) \le \frac{1}{4}$.
\end{theorem}

\begin{proof}
We recall that
\begin{align*}
F_{\tau_{k+1}}(\mu) = r \tau_k \Big( S_t(\mu) + R(\mu) \Big) + B(\mu) = r F_{\tau_k}(\mu) - (r - 1)B(\mu)
\end{align*}

Taking the gradient yields $\nabla F_{\tau_{k+1}}(\mu) = r \nabla F_{\tau_k}(\mu) - (r - 1)\nabla B(\mu)$.
Furthermore, because $S_t+R$ is convex, we have the two following inequalities:
\begin{align*}
\nabla^2 F_{\tau_{k+1}}(\mu) \succeq \nabla^2 F_{\tau_k}(\mu) \quad \text{and} \quad \nabla^2 F_{\tau_{k+1}}(\mu) \succeq \nabla^2 B(\mu)
\end{align*}
Let $(\Delta \mu_k,s_k)$ be the solution of the KKT system $(\mu_k, F_{\tau_{k+1}})$ \Cref{def-KKT}. The intermediate Newton decrement can be rewritten
\begin{align*}
\delta_{\text{shift}}^2 & : = \delta(\mu_k, F_{\tau_{k+1}})^2 \\
&  = \Delta \mu_k^\top \nabla^2 F_{\tau_{k+1}}(\mu_k) \Delta \mu_k \\
& = -\Delta \mu_k^\top \nabla F_{\tau_{k+1}}(\mu_k) = \underbrace{r \left( -\Delta \mu_k^\top \nabla F_{\tau_k}(\mu_k) \right)}_{\text{Term A}} + \underbrace{(r - 1)\Delta \mu_k^\top \nabla B(\mu_k)}_{\text{Term B}}
\end{align*}

On a hand, let $(\nu_k,s_k)$ be the solution of the KKT system with $(\mu_k, F_{\tau_k})$ \Cref{def-KKT}. By Cauchy-Schwarz,
\begin{align*}
-\Delta \mu_k^\top \nabla F_{\tau_k}(\mu_k) & = \Delta \mu_k^\top \nabla^2 F_{\tau_k}(\mu_k) v_k  \le \Vert{}\Delta \mu_k\Vert{}_{\nabla^2 F_{\tau_k}} \Vert{}v_k\Vert{}_{\nabla^2 F_{\tau_k}}
\end{align*}
By the assumption that the previous step was centered, 
\begin{align*}
\Vert{}v_k\Vert{}_{\nabla^2 F_{\tau_k}} & = \delta(\mu_k, F_{\tau_k}) \le \beta \\
\Vert{}\Delta \mu_k\Vert{}_{\nabla^2 F_{\tau_k}} & \le \Vert{}\Delta \mu_k\Vert{}_{\nabla^2 F_{\tau_{k+1}}} = \delta_{\text{shift}}.
\end{align*}
Combining all the equation, the Term $A$ is smaller than $ r \beta \delta_{\text{shift}}$.

On the other hand, we have 
\begin{align*}
\Delta \mu_k^\top \nabla B(\mu_k) \le \Vert{}\Delta \mu_k\Vert{}_{\nabla^2 B} \Vert{}\nabla B(\mu_k)\Vert{}^*_B
\end{align*}

Because $\nabla^2 B \preceq \nabla^2 F_{\tau_{k+1}}$, the first term is bounded by $\delta_{\text{shift}}$. By the fundamental property of a $J$-self-concordant barrier (\cite{nesterov2018lectures} Definition 5.3.1), its local dual norm is globally bounded: $\Vert{}\nabla B(\mu_k)\Vert{}^*_B \le \sqrt{J}$. Thus, $\text{Term B} \le (r-1)\sqrt{J}\delta_{\text{shift}}$.

Substituting both terms back into the expansion gives
\begin{align*}
\delta_{\text{shift}} & \le r\beta + (r-1)\sqrt{J} \\
 & = \left(1 + \frac{\gamma}{\sqrt{J}}\right)\beta + \gamma \le \beta + \gamma(1 + \beta)
\end{align*}
Setting $\beta = 1/4, \gamma = 1/15$, we obtain $\delta_{\text{shift}} \le 1/3 < 1$. Because $F_{\tau_{k+1}}$ is strictly $1$-self-concordant, and the intermediate error is safely $< 1$, classical single-step Newton convergence (\Cref{thm:constrained_newton}) applies. Exactly one normal, undamped Newton step drops the decrement quadratically:
\begin{align*}
\delta(\mu_{k+1}, F_{\tau_{k+1}}) \le \left(\frac{\delta_{\text{shift}}}{1 - \delta_{\text{shift}}}\right)^2 \le \left(\frac{1/3}{1 - 1/3}\right)^2 = \frac{1}{4} = \beta
\end{align*}
\end{proof}

\subsection{Overall complexity}
We recall the result to prove :
\begin{corollary}
Suppose that $\| w_{\star} - w_1\|_2 \leq D$.  By fixing $\beta = \frac{1}{4}$, $\epsilon_t := \frac{L}{t^4}$, under \Cref{Assumption-linear-convex}, we have
\begin{align*}
N_{\text{Phase I}} + N_{\text{Phase II}} = \mathcal{O} \left( G + \sqrt{J} \log \Big( J G t \Big)  \frac{G^2 }{L} +  2 G \sqrt{D^2 + \log(t)} + S_\star  \right)
\end{align*}
\end{corollary}
\begin{proof}
First, we notice that 
\[  N_{\text{Phase II}} = \mathcal{O} \left( G + \sqrt{J} \log \Big( J G t \Big) \right) \]
In order to show the upper bound of the Phase I complexity, we recall that \Cref{eq-main-acceleration-results} implies for any $t \geq 1$
\begin{equation}
\begin{aligned}
A_t  \max_{\lambda \in \Delta }\Phi(\lambda ,w_{t+1}^{ag}) & \leq \frac{1}{2}\| w_{\star} - w_1 \|_2^2 -  \frac{1}{2} \| w_{\star} - w_{t+1} \|_2^2   + \sum_{\tau=1}^{t} \alpha_t  \Phi(\lambda_{t+1},w_\star)  + \sum_{\tau=1}^{t} A_\tau \epsilon_\tau \\
 & \leq \frac{1}{2}\| w_{\star} - w_1 \|_2^2 -  \frac{1}{2} \| w_{\star} - w_{t+1} \|_2^2   + A_t \max_{\lambda \in \Delta } \Phi(\lambda,w_\star)  + \sum_{\tau=1}^{t} A_\tau \epsilon_\tau 
\end{aligned}
\end{equation}
Therefore, since $w_\star$ is the minimum by definition, we deduce that
\begin{align*}
 \| w_{\star} - w_{t+1} \|_2^2 \leq \| w_{\star} - w_1 \|_2^2  + 2 \sum_{\tau=1}^{t} A_\tau \epsilon_\tau  \leq D^2 + \mathcal{O}(\log(t))
\end{align*}
We also observe that $w_t^{md}$ is a convex combination of $w_t$, therefore for all $t\geq 1$
\begin{align*}
\| w_\star - w_t^{md}\|_2 \leq \max_{1 \leq \tau \leq t } \| w_\star - w_{\tau} \|_2
\end{align*}
\Cref{thm-phase1} implies that
\begin{align*}
N_{\text{Phase I}} = \mathcal{O} \left(  \frac{G^2 }{L} +  2 G \sqrt{D^2 + \log(t)} + S_\star \right)
\end{align*}
which concludes our proof.
\end{proof}

\section{Spectrahedron}

The \textit{spectrahedron} is defined as the set of $J \times J$ symmetric PSD matrices with unit trace 
\[ \mathcal{S}_J := \{ \lambda \in \mathbb{S}_+^J \mid \text{Tr}(\lambda) = 1 \}.\] 
The affine coupling for matrices becomes :
\begin{align*}
\min_{w \in \mathbb{R}^d } \max_{\lambda \in S_J}    \text{Tr}(\lambda F(w)) - R(\lambda)
\end{align*}
where $F :\mathbb{R}^d \xrightarrow{} S_J $ is a matrix-convex (Loewner-convex) and $R :\mathbb{R}^J \xrightarrow{} \mathbb{R}$ convex.

For the spectrahedron we use the log-determinant barrier:
\begin{equation}
\begin{aligned}
B_{\Lambda}(\lambda) &:= -\log \det(\lambda) \\
F_{\tau}(\lambda) & := \tau (S_t(\lambda) + R(\lambda))  -\log \det(\lambda)  
\end{aligned}
\end{equation}
This function is strictly $J$-self-concordant for the PSD cone $\mathbb{S}^J_+$ \citep{nesterov2018lectures}. 

\section{Simulations}
\label{sec-appendix-simu}
Note that data generation and the pre-computation of the smoothness constant $L$ are excluded from the runtime measurements. Because the exact optimal primal value is intractable, we approximate it empirically using the lowest objective value achieved across all runs of our proposed method. 
\textbf{Data generation}
\begin{itemize}
    \item A true underlying weight vector: $w_{true} \sim \mathcal{N}\left(0, \frac{20}{d}I\right)$
    \item Group-specific distribution shift: $\mu_j \sim \mathcal{N}\left(0, \frac{1}{d}I\right)$ 
    \item Sample generation: $x_{j,i} \sim \mathcal{N}\left(0, \frac{1}{d}I\right) + 0.1 \mu_j$
    \item Label Generation: $y_{j,i} = \text{sign}(\langle x_{j,i}, w_{true} \rangle + \varepsilon_{j,i})$ where $\varepsilon_{j,i} \sim \mathcal{N}(0, 1)$ 
\end{itemize}

\textbf{Gradient computation}
\begin{align*}
\nabla f_j(w) & = \frac{1}{N} \sum_{i=1}^{N} \frac{-y_{j,i} X_{j,i}}{1 + \exp(y_{j,i} X_{j,i}^\top w)} \\
L_j & :=  \lambda_{\max}\left( \frac{1}{4N} X_j^\top X_j \right) \\ 
L & := \max_{j} L_j \\
\end{align*}

\textbf{Algorithm configuration}
Our algorithm is detailed in \Cref{alg:accelerated_method_quadratic}. In this experiment, we compare our approach against the accelerated proximal method, the optimistic method, and the extragradient method, defined as follows:

\textbf{Accelerated Proximal Gradient Descent}
When $\gamma > 0$, the exact maximization over $\lambda$ yields a strongly-smoothed primal function $g(w) = \max_{\lambda \in \Delta_J} \Phi(w, \lambda)$. Compute
\begin{align*}
\lambda_t^\star & = \Pi_{\Delta_J} \left( \textbf{1}_J + \frac{f(y_{t-1})}{\gamma} \right) \\
L_{\text{local}} & = L + \frac{\Vert{} \nabla f(y_{t-1})\Vert{}_2^2}{\gamma} \\
w_t & = y_{t-1} - \frac{1}{L_{\text{local}}} \nabla f(y_{t-1}) \lambda_t^{\star} 
\end{align*}
Apply Nesterov Momentum Update:
\begin{align*}
\tau_t & = \frac{1 + \sqrt{1 + 4\tau_{t-1}^2}}{2} \\
y_t & = w_t + \frac{\tau_{t-1} - 1}{\tau_t} (w_t - w_{t-1})
\end{align*}

\textbf{Extragradient}
Compute the step-size and midpoint
\begin{align*}
\eta_t & = \frac{1}{\max(L, \gamma) + \Vert{}\nabla f(w_{t-1})\Vert{}_2} \\
w_{\text{mid}} & = w_{t-1} - \eta_t \nabla f(w_{t-1})\lambda_{t-1} \\
\lambda_{\text{mid}} & = \Pi_{\Delta_J} \left[ \lambda_{t-1} + \eta_t \nabla_\lambda \Phi(w_{t-1}, \lambda_{t-1}) \right]
\end{align*}

Compute the final update
\begin{align*}
w_t & = w_{t-1} - \eta_t \nabla f(w_{\text{mid}}), \lambda_{\text{mid}} \\
\lambda_t & = \Pi_{\Delta_J} \left[ \lambda_{t-1} + \eta_t \nabla_\lambda \Phi(w_{\text{mid}}, \lambda_{\text{mid}}) \right]
\end{align*}

\end{document}

%% file: math_commands.tex
\usepackage{amsmath,amsfonts,bm}

\def\1{\bm{1}}

\DeclareMathAlphabet{\mathsfit}{\encodingdefault}{\sfdefault}{m}{sl}
\SetMathAlphabet{\mathsfit}{bold}{\encodingdefault}{\sfdefault}{bx}{n}

